\documentclass[leqno, 12pt]{amsart}

\usepackage{seqsplit}
\usepackage[english]{babel}

\usepackage[utf8]{inputenc}
\usepackage{geometry}

\usepackage{lineno}

\usepackage{amsmath}
\usepackage{amssymb}
\usepackage{amsthm}

\usepackage{graphicx}
\usepackage{subfig}
\usepackage{caption}
\usepackage{hyperref}
\usepackage{float}
\usepackage[table]{xcolor}
\usepackage{booktabs}
\usepackage[foot]{amsaddr}

\usepackage[capitalise, noabbrev]{cleveref}

\DeclareFixedFont{\ttb}{T1}{txtt}{bx}{n}{10} % for bold
\DeclareFixedFont{\ttm}{T1}{txtt}{m}{n}{10}  % for normal

\usepackage{color}
\definecolor{deepblue}{rgb}{0,0,0.5}
\definecolor{deepred}{rgb}{0.6,0,0}
\definecolor{deepgreen}{rgb}{0,0.5,0}

\usepackage{listings}

\newcommand\pythonstyle{\lstset{
		language=Python,
		basicstyle=\ttm,
		morekeywords={self},              % Add keywords here
		keywordstyle=\ttb\color{deepblue},
		commentstyle =\ttm,
		emph={MyClass,__init__},          % Custom highlighting
		emphstyle=\ttb\color{deepred},    % Custom highlighting style
		stringstyle=\color{deepgreen},
		frame=tb,                         % Any extra options here
		showstringspaces=false
}}

\lstnewenvironment{python}[1][]
{
	\pythonstyle
	\lstset{#1}
}
{}

\newcommand{\Z}{\mathbb Z}

\newcommand{\Ss}{\mathfrak S}

\newcommand{\Pos}{\mathcal P}
\newcommand{\RPos}{\mathcal R}
\renewcommand{\L}{\mathbf L}

\newcommand{\Sort}{\mathbf{Sort}}
\newcommand{\sort}{\operatorname{sort}}

\newcommand{\Tree}{T}
\newcommand{\PTree}{U}
\newcommand{\Trees}{\operatorname{T}}
\newcommand{\IdL}{\mathbf{Low}}
\newcommand{\JIrr}{\mathcal J}
\newcommand{\word}{\mathrm{w}}

\newcommand{\Cat}{\operatorname{Cat}}

\newcommand{\pack}{\operatorname{pack}}

\newcommand{\plusone}{\! + \! 1}

\newcommand{\minusone}{\!-\! 1}

\definecolor{darkblue}{rgb}{0,0,0.7} % darkblue color
\newcommand{\darkblue}{\color{darkblue}} % darkblue command
\newcommand{\defn}[1]{\textsl{\darkblue #1}} % emphasis of a definition

\newcommand{\GT}{\mathbf{GT}}

\numberwithin{equation}{section}
\newtheorem{thm}{Theorem}[section]
\newtheorem{lem}[thm]{Lemma}
\newtheorem{prop}[thm]{Proposition}
\newtheorem{nota}[thm]{Notation}
\newtheorem{cor}[thm]{Corollary}
\newtheorem{defi}[thm]{Definition}
\newtheorem{cnj}[thm]{Conjecture}

\theoremstyle{remark}
\newtheorem{rmk}[thm]{Remark}

\title{Middle orders: all distributive lattices \linebreak between weak and Bruhat}

\author{Ludovic Schwob$^1$}
\address{$^1$Univ Gustave Eiffel, CNRS, LIGM, F-77454 Marne-la-Vallée, France}

\begin{document}
	\begin{abstract}
		For a given Coxeter group, we study distributive lattices called middle orders refining the weak order and refined by the Bruhat order. In type $A$, we construct such lattices indexed by binary trees using a direct bijection between permutations and lower sets of a certain partition of the root poset into rectangles.  When the binary tree is a left-comb tree, we recover the middle order defined by Bouvel, Ferrari, and Tenner (2025). We study combinatorial properties of these lattices, and show they are the only distributive lattices between the weak and Bruhat orders in type $A$. For general Coxeter groups, we study middle orders on parabolic quotients and use these to generalize our construction in type $A$ to other Weyl groups, obtaining so-called ``minuscule middle orders''. We show that they are a subset of sorting orders defined by Armstrong (2009), and we give conjectural descriptions of all middle orders that are not minuscule.
		
			%	We construct distributive lattices on $\Ss_n$ refining the weak order on permutations and refined by the Bruhat order, generalizing the middle order defined by Bouvel, Ferrari and Tenner. These lattices, which we also call middle orders, are defined using a direct bijection between permutations and lower sets of a certain partition of the root poset into rectangles. We study combinatorial properties of these lattices, and show they are the only distributive lattices between the weak and Bruhat orders. We also consider generalizations of middle orders to other finite Coxeter groups and to parabolic quotients of finite Coxeter groups. We finally define minuscule middle orders on Weyl groups as a generalization of our construction on $\Ss_n$ and show that they are sorting orders.
		%and we introduce Jurassian lattices, which are subposets of minuscule middle orders sharing many properties with Cambrian lattices.
	\end{abstract}

	\maketitle
	
	\tableofcontents
	\section{Introduction}
%	Many lattices on permutations are known to be refinements of the weak order coarsenings of the Bruhat order.
	The weak order and the Bruhat order are two well-studied posets on permutations, which are also defined on Coxeter groups in general~\cite{BjBr05}. The weak order is a semidistributive lattice, whose quotients include many interesting lattices such as the Tamari lattice, Cambrian, and biCambrian lattices~\cite{BaRe17,Read06}, the lattice of permutrees~\cite{PiPo17}, and the lattice of Baxter permutations~\cite{Read05bis}. These quotients can all be realized as polytopes~\cite{PiSa19}. The Bruhat order is a refinement of the weak order with the same ranking. This poset appeared first in the study of Schubert varieties~\cite{Ehr34}. Lascoux and Schützenberger showed that the completion of the Bruhat order on permutations is the lattice of alternating sign matrices~\cite{LaSc96}, which is distributive.
	Distributive lattices are essential objects in lattice theory and in combinatorics. One fundamental result on distributive lattices is the Birkhoff's representation theorem~\cite{Birk37}, which states that every finite distributive lattice can be represented as the lattice of lower sets of a poset.
	Notable examples include the lattice of congruences of any lattice, boolean algebras, the Young lattice, the lattice of alternating sign matrices, and lattices of $d$-factors~\cite{Prop25}, including many lattices on tilings and matchings.
	
	A lattice on permutations called ``middle order'' has been recently described by Bouvel, Ferrari, and Tenner \cite{BFT24}. This lattice is defined by the coordinatewise comparison of Lehmer codes, and is distributive since the join and meet operations are given by coordinatewise maximum and minimum. The middle order is a refinement of the weak order and is refined by the Bruhat order. In~\cite{ClOv25}, Claussen and Ovenhouse noticed that a distributive lattice on perfect matchings of certain graphs could be obtained by restricting the middle order to alternating permutations. Other lattices with this property have been described, such as the sorting orders described by Armstrong \cite{Arm09}, which are join-distributive lattices on finite Coxeter groups. In~\cite{CFY25}, Campanini, Fedele, and Yıldırım constructed lattices of pretorsion classes that are sometimes distributive and are conjecturally between the weak and Bruhat orders.
	
	In this paper we focus on distributive lattices that are refinements of the weak order and coarsenings of the Bruhat order, which we will call middle orders in general. We give a combinatorial construction of $C_{n-1}$ (\emph{i.e.}, the $n \! - \! 1$-st Catalan number) middle orders on $\Ss_n$, each corresponding to a binary tree with $n$ leaves. For this we construct a direct bijection between permutations and lower sets of a disjoint union of rectangular posets. This bijection can be seen as packing of inversion sets. We show that these lattices are the only middle orders on $\Ss_n$ by describing compatibility relations that covering relations of the Bruhat order must respect when they appear in a middle order. The proof involves a bijection between middle orders and Gelfand-Tsetlin patterns.
	We also investigate middle orders in other Coxeter groups, in particular in Weyl groups and parabolic quotients of finite Coxeter groups. We show that middle orders on multipermutations can be constructed in a similar way as those on $\Ss_n$. We describe particular middle orders on Weyl groups that we call minuscule middle orders, generalizing our construction on $\Ss_n$. 
	 We show that minuscule middle orders can be obtained as distributive sorting orders. This construction of minuscule middle orders involves fully commutative elements and minuscule heaps as described by Stembridge~\cite{Stem96}. We give a list of other distributive sorting orders, which we conjecture to be complete in type $B$, $C$, and $D$, and also give an exhaustive list for $F_4$ and $H_3$.
 %   We eventually study middle orders on affine Coxeter groups, such as the Young lattice described by Björner and Brenti~\cite{BjBr95}.
	
	The rest of this article is organized as follows. In Section~\ref{sec:notations} we explain notations we use later for distributive lattices, Coxeter groups and Weyl groups. In Section~\ref{sec:construction}, we explain the construction of middle orders on $\Ss_n$.
	In Section~\ref{sec:exhaustivity}, we present our main result, by proving that we built all middle orders on $\Ss_n$.
	In Section~\ref{sec:properties}, we study middle orders on $\Ss_n$ up to isomorphism, and give formulas for the numbers of intervals, chains, and boolean chains in middle orders. 
	We consider middle orders on multipermutations in Section~\ref{sec:parabolic}. 
	In Section~\ref{sec:minuscule}, we describe minuscule middle orders on Weyl groups, and in Section~\ref{sec:sorting}, we study distributive sorting orders on finite Coxeter groups.
    %, and in Section~\ref{sec:affine}, we study middle orders on affine Coxeter groups.
	
	\section{Notations and preliminaries}	
	\label{sec:notations}
	\subsection{Distributive lattices}
		\label{sec:distri}
	A lattice $\L$ is \emph{distributive} if and only if its join and meet distribute over each other:
	$$\forall x,y,z\in \L, \ x\wedge (y\vee z)=(x\wedge y)\vee (x\wedge z)\text{ and } x\vee (y\wedge z)=(x\vee y) \wedge (x\vee z).$$
	Let $\Pos:=\JIrr(\L)$ the poset of its \defn{join-irreducibles}, \emph{i.e.} elements of $\L$ covering exactly one element (see~\cite[Chapter IX]{Birk48} or \cite[Chapter II]{Grae11} for more details). $\L$ is isomorphic to the lattice $\IdL(\Pos)$ of lower sets of $\Pos$ ordered by inclusion. Dually, it is also isomorphic to the lattice $\mathbf{Upp}(\Pos)$ of upper sets of $\Pos$ ordered by reverse inclusion.
	This result, known as Birkhoff's representation theorem (or the fundamental theorem of distributive lattices), is a powerful tool for studying distributive lattices, as many of their properties can be read on their posets of irreducibles.

	\begin{itemize}
	\item The product of finite distributive lattices correspond to the disjoint union of their posets of join-irreducible elements: $\IdL(\Pos)\times \IdL(\mathcal Q)\cong \IdL(\Pos\sqcup \mathcal Q)$;
	\item Quotienting $\L$ by a lattice congruence yields a quotient isomorphic to $\IdL(\mathcal Q)$ where $\mathcal Q$ is a subposet of $\Pos$;
	\item A sublattice of $\L$ sharing the same minimal and maximal elements as $\L$ is isomorphic to $\IdL(\RPos)$ where $\RPos$ is a refinement of $\Pos$;
	\item An interval of $\L$ is isomorphic to $\IdL(\Pos\setminus (L\cup U))$ where $L\in \IdL(\Pos)$ and $U\in \mathbf{Upp}(\Pos)$ are disjoint.
	\end{itemize}
	We will see in Section~\ref{sec:chains} that complete chains (resp. complete boolean chains) of $\L$ of length $\ell$ are in bijection with surjective increasing (resp. strictly increasing) functions from $\IdL(\L)$ to $\{1,\dots, \ell\}$.
	%The order dimension of a distributive lattice $\L$ can also be easily computed, as it is equal to the length of the longest antichain in $\JIrr(\L)$.

	\subsection{Coxeter groups}
	
		Let $(W,S)$ be a Coxeter system (see \cite{BjBr05} for basic notions). The \defn{length} $\ell(w)$ of $w\in W$ is the number of generators in a reduced expression of $w$. The \defn{Bruhat order} on $W$ is defined by the inclusion of reduced words, \emph{i.e.}, $v\le w$ if some reduced expression of $v$ is a subword of some reduced expression of $w$. The right (resp. left) \defn{weak order }on $W$ is defined by $v\le_R w$ if $v$ is a prefix of some reduced expression of $w$ (resp. $v\le_L w$ if $v$ is a suffix of $w$). If $W$ is finite (see Figure \ref{fig:coxeter} for the classification of finite Coxeter groups), the weak order on $W$ is a lattice and it is distributive only if $W$ is commutative. The Bruhat order on $W$ is also a lattice only if $W$ is commutative, which is not the case for most Coxeter groups.
		
		Let $J$ be a subset of $S$. The \defn{parabolic subgroup} $W_J$ of $W$ is the subgroup generated by the elements of $J$, and its \defn{parabolic quotient} $^JW$ is the set of minimal elements of cosets in $W_J\backslash W$.
		Cosets in $W_J\backslash W$ are intervals of the left weak order (though we will consider later the right weak order on $W_J$), and are also intervals of the Bruhat order. The quotient $^JW$ is an interval of the right weak order, but is not an interval of the Bruhat order in general.
		
		\subsection{Weyl groups} A \defn{Weyl group} is a finite Coxeter group $(W,S)$ acting on a root system $\Phi$, which can be partitioned into \defn{positive and negative roots} $\Phi^+\sqcup \Phi^-$, with $\Phi^-=- \Phi^+$. The set of positive roots $\Phi^+$ is ordered by $\alpha \le \beta$ if and only if $\beta-\alpha$ is a nonnegative linear combination of positive roots. This order is called the \defn{root poset} of $W$. Its minimal elements are called \defn{simple roots} and correspond to generators of $W$. More generally, each positive root $\alpha$ corresponds to a \defn{reflection }$r\in W$ (\emph{i.e.}, any conjugate of a generator), $r$ acting on $\Phi$ as the orthogonal reflection mapping $\alpha$ onto $r(\alpha)=-\alpha$.
		Some roots of particular interest are the \defn{minuscule roots}, as they intervene in the study of minuscule representations of semisimple Lie algebras. They can be defined as the simple roots appearing with multiplicity $1$ when writing the maximal element of $\Phi^+$ as a nonnegative linear combination of simple roots.
		
		\begin{figure}[H]
		\centering
		\includegraphics[width=0.45\linewidth]{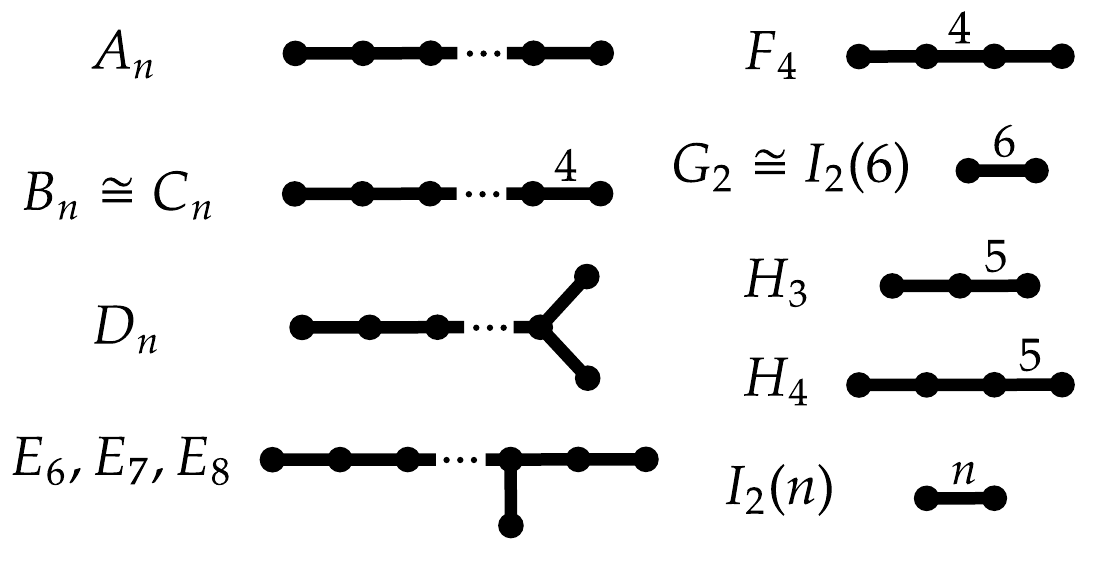}\qquad
		\includegraphics[width=0.45\linewidth]{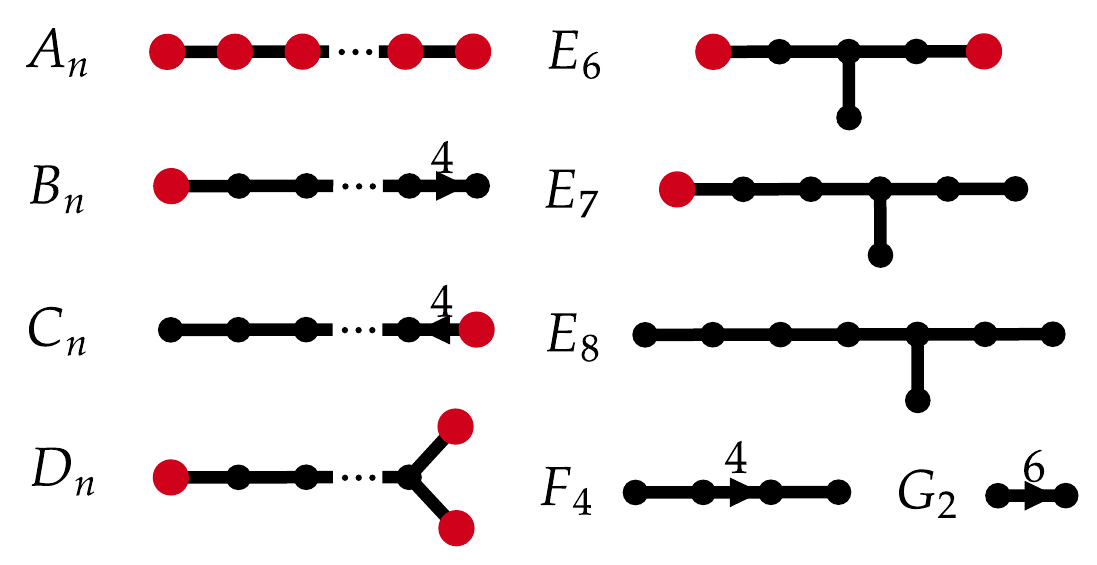}
		\caption{On the left: Coxeter diagrams of finite Coxeter groups. On the right: Dynkin diagrams of Weyl groups, with highlighted nodes corresponding to minuscule roots.}
		\label{fig:coxeter}
		\end{figure}
		
		The \defn{inversion set} of $w\in W$ is $N(w):=\Phi^+\cap w(\Phi^-)$. $N(w)$ characterizes $w$ and $|N(w)|=\ell(w)$. Inversion sets are the subsets $A\subset \Phi^+$ that are both closed and coclosed, where $A$ is \defn{closed} if for all $\alpha,\beta \in A$, $\alpha+\beta\in \Phi^+$ implies $\alpha+\beta\in A$, and $A$ is \defn{coclosed} if $\Phi^+\setminus A$ is closed. We have $v\le_Rw$ if and only if $N(v)\subset N(w)$.
	
	\section{Construction of the middle orders on $\Ss_n$}
	\label{sec:construction}
	
	We construct the posets of irreducible elements of middle orders by partitioning a triangular poset into rectangular posets each containing exactly one minimal element. For this we choose minimal elements in a given order, removing at each step the elements remaining above.
	\begin{defi}
		The poset $\Pos_n$ is the set of pairs of integers $(i,j)$ such that $1\le i<j\le n$, with order defined by $(i_1,j_1)< (i_2,j_2)\iff i_2\le i_1$ and $j_1\le j_2$.
		Let $\sigma\in \Ss_{n-1}$. For all $1\le k\le n-1$, we define
		$$\RPos_{\sigma,k}:=\{x\in \Pos_n:(k,k+1)<x\text{ and }\forall \ell<\sigma^{-1}(k), (\ell,\ell+1)\nless x\}.$$
		The $(\RPos_{\sigma,k})_{1\le k\le n-1}$ form a partition of $\Pos_n$. Each $\RPos_{\sigma,k}$ inherits a poset structure from $\Pos_n$. We define the poset $\RPos_\sigma$ as the disjoint union of the posets $\RPos_{\sigma,k}$.
	\end{defi}

We now characterize permutations which give the same partition of $\Pos_n$.
	\begin{prop}
		Let $\sigma,\tau\in \Ss_n$. We have $\RPos_\sigma=\RPos_\tau$ if and only if $\sigma\equiv \tau$, \emph{i.e.}, they can be obtained from each other by rewritings of the form $b\cdots ac\equiv b\cdots ca$ with $a<b<c$.
		%, or equivalently if the binary search trees of $\sigma$ and $\tau$ are the same.
	\end{prop}
\begin{proof}
	Removing the rectangular portion above $(b,b+1)$ splits $\RPos_n$ into two disconnected left and right parts containing respectively $(a,a+1)$ and $(c,c+1)$, so removing the elements above $(a,a+1)$ and then those above $(c,c+1)$ gives the same result as removing elements above $(c,c+1)$ and then those above $(a,a+1)$. Hence, $\RPos_\sigma=\RPos_\tau$ if $\sigma$ and $\tau$ can be obtained from each other by rewritings of the form $b\cdots ac\equiv b\cdots ca$ with $a<b<c$.
	As $\equiv$ is the mirror of the sylvester congruence~\cite{HNT05}, permutations in $\Ss_{n-1}$ up to $\equiv$ are counted by $C_{n-1}$, and an immediate recursion shows that there is the same number of distinct rectangulations.
	%Conversely, there is a bijection between the posets $\RPos_\sigma$ with $\sigma\in \Ss_{n-1}$ and binary trees with $n$ leaves, so it is clear that we have $C_{n-1}$ equivalence classes of permutations.
\end{proof}

 The \defn{binary search tree} $\Tree$ of a permutation $\sigma=\sigma_1\dots \sigma_n$ is the binary tree whose root is $\sigma_1$, whose left (resp. right) subtree is the binary subtree of the subword of $\sigma$ consisting in values strictly smaller (resp. greater) than $\sigma_1$.
  We have $\sigma\equiv \tau$ if and only if their binary search trees are the same~\cite{HNT05}, hence $\RPos_\sigma$ is entirely determined by $\Tree$ (see Figure~\ref{fig:moposet5}). We will write $\RPos_\Tree:=\RPos_\sigma$, and we denote by $\Trees_n$ the set of binary trees with internal nodes $1,2,...,n$.
 
 \begin{figure}[H]
 	\centering
 	\includegraphics[width=0.6\linewidth]{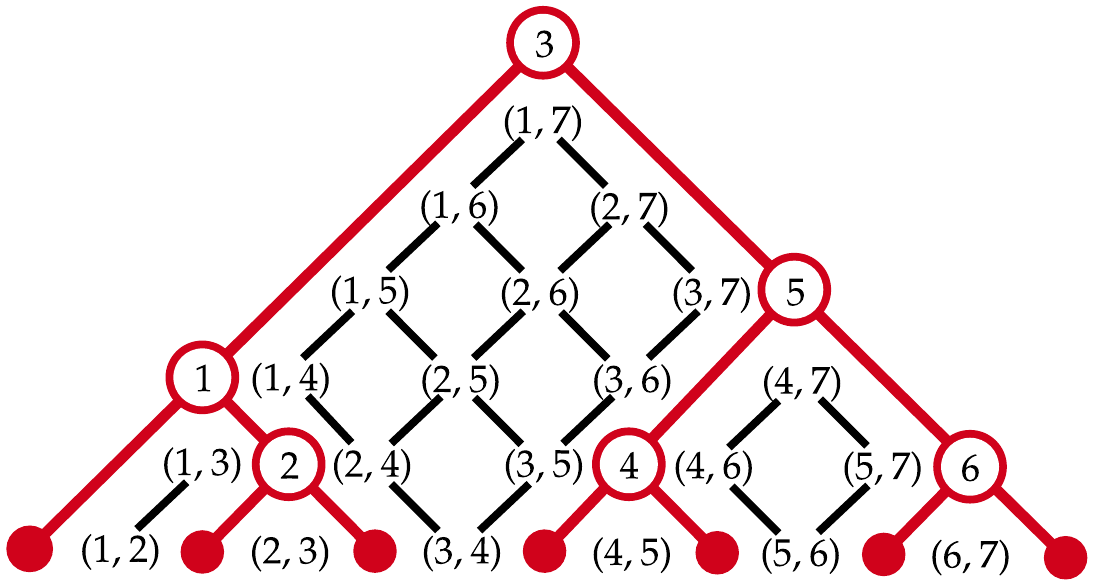}
 	\caption{The poset $\RPos_{312546}$ with its corresponding binary search tree.}
 	\label{fig:moposet5}
 \end{figure}

We will now show that for all $T\in \Trees_{n-1}$, we have a bijection between $\Ss_n$ and lower sets of $\RPos_\Tree$.

\begin{defi}
   \label{def:ideal}
	Let $\Tree\in \Trees_{n-1}$. For all $1\le k\le n-1$, let $(i_{\Tree,k},j_{\Tree,k}):=\max \RPos_{\Tree,k}$. For all $w\in \Ss_n$, let $w^{L,k}$ and $w^{R,k}$ be the subwords of $w$ consisting respectively of values between $i_{\Tree,k}$ and $k$, and between $k+1$ and $j_{\Tree,k}$.
	
		We define $I_\Tree(w)$ as the subset of $\RPos_\Tree$ such that for all $1\le k\le n-1$, $i_{\Tree,k}\le i\le k$ and $k+1\le j\le j_{\Tree,k}$, $I_\Tree(w)$ contains $(i,j)$ whenever $w^{L,k}_{i-i_{\Tree,k}+1}$ appears after $w^{R,k}_{j-k}$ in $w$.
\end{defi}

It is straightforward to check that $I_\Tree(w)$ is a lower set of $\RPos$. In fact it is possible to obtain $I_\Tree(w)$ by ``packing'' the inversion set of $w$ to the left and then to the right, or \emph{vice versa}.

\begin{defi}
\label{def:pack}
   Let $\Tree\in \Trees_{n-1}$ and $X\subset \RPos_{\Tree}$. We say that $X$ is $\Tree$-\defn{left-packed} if for all $1\le i<j<n$, $(i,j)< (i,j+1)$ in $\RPos$ and $(i,j+1)\in X$ imply $(i,j)\in X$. We say that $X$ is $\Tree$-\defn{right-packed} if for all $1< i<j\le n$, $(i,j)< (i-1,j)$ in $\RPos$ and $(i-1,j)\in X$ imply $(i,j)\in X$. 
   
   We define $L_\Tree(X)$ as the only subset of $\RPos$ which is left-packed and has the same number of elements as $X$ of the form $(i,j)$ with $k+1\le j\le j_{\Tree,k}$ for all $1\le k\le n-1$ and $i_{\Tree,k}\le i\le k$.
   
   We define $R_\Tree(X)$ as the only subset of $\RPos$ which is right-packed and has the same number of elements as $X$ of the form $(i,j)$ with $i_{\Tree,k}\le i\le k$ for all $1\le k\le n-1$ and $k+1\le j\le j_{\Tree,k}$.
\end{defi}
As we can see in~\cref{fig:bij1}, $L_\Tree(X)$ and $R_\Tree(X)$ correspond respectively to sliding elements of $X$ to the lower left or lower right.
\begin{prop}
   \label{prop:packing}
   Let $\Tree\in \Trees_{n-1}$ and $w\in \Ss_n$. Let $X$ be the inversion set of $w$, \emph{i.e.}, the set of pairs $(i,j)$ with $1\le i<j\le n$ such that $j$ appears before $i$ in $w$. We have the equalities $$I_\Tree(w)=L_\Tree(R_\Tree(X))=R_\Tree(L_\Tree(X)).$$
\end{prop}

\begin{proof}
  The following characterizations of $L_\Tree(X)$ and $R_\Tree(X)$ follows from Definition~\ref{def:pack}:
   $L_\Tree(X)$ (resp. $R_\Tree(X)$) is the subset of $\RPos_\Tree$ such that for all $1\le k\le n-1$, $i_{\Tree,k}\le i\le k$ and $k+1\le j\le j_{\Tree,k}$, containing $(w^{L,k}_{i-i_{\Tree,k}+1},j)$ (resp. $(i,w^{R,k}_{j-k})$) whenever $w^{L,k}_i$ appears after $w^{R,k}_{j-k}$ in $w$. It follows that $L_\Tree(R_\Tree(X))$ and $R_\Tree(L_\Tree(X))$ contain $(i,j)$ if and only if $w^{L,k}_{i-i_{\Tree,k}+1}$ appears after $w^{R,k}_{j-k}$ in $w$, hence they are both equal to $I_\Tree(X)$.
\end{proof}

\begin{figure}
\centering
\includegraphics[width=0.4\linewidth]{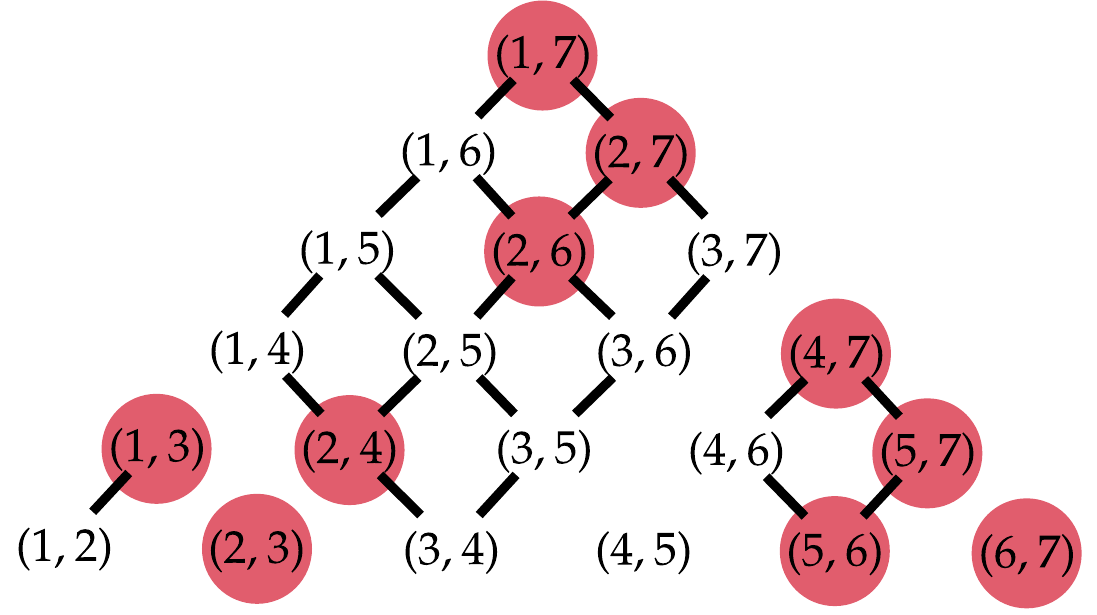}\raisebox{50pt}{$\longrightarrow$}
\includegraphics[width=0.4\linewidth]{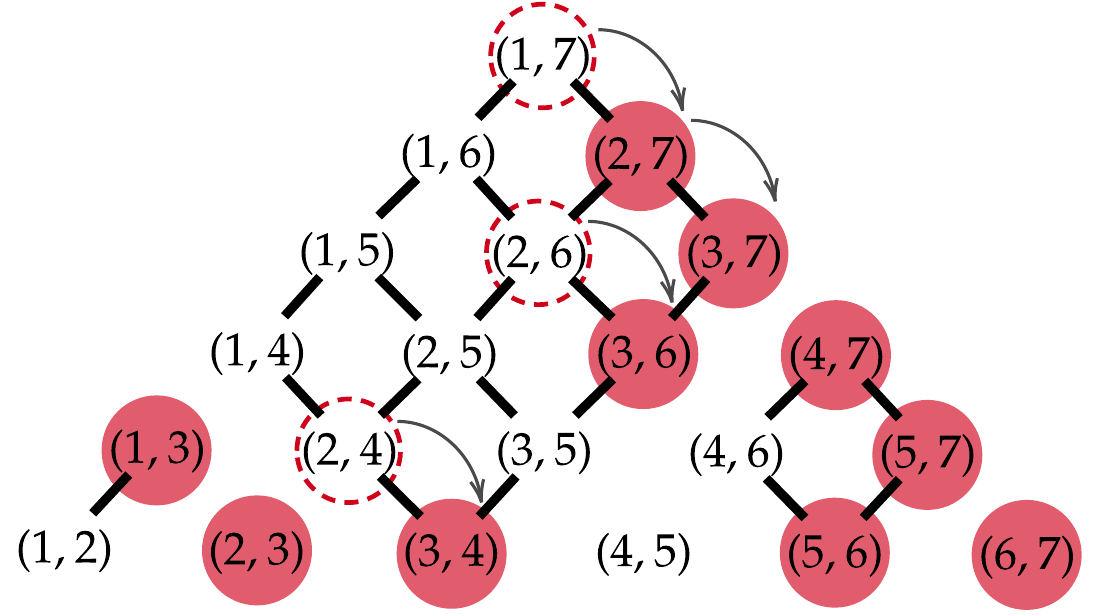}

$\downarrow$ \hspace{250pt} $\downarrow$ 

\raisebox{3pt}{\includegraphics[width=0.4\linewidth]{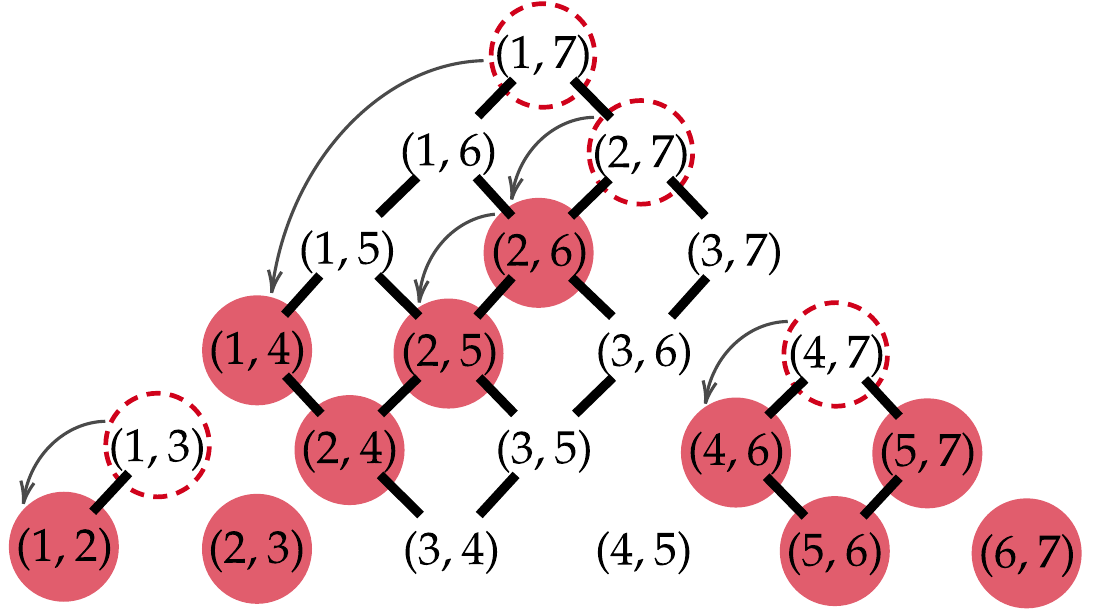}}\raisebox{50pt}{$\longrightarrow$}
\includegraphics[width=0.4\linewidth]{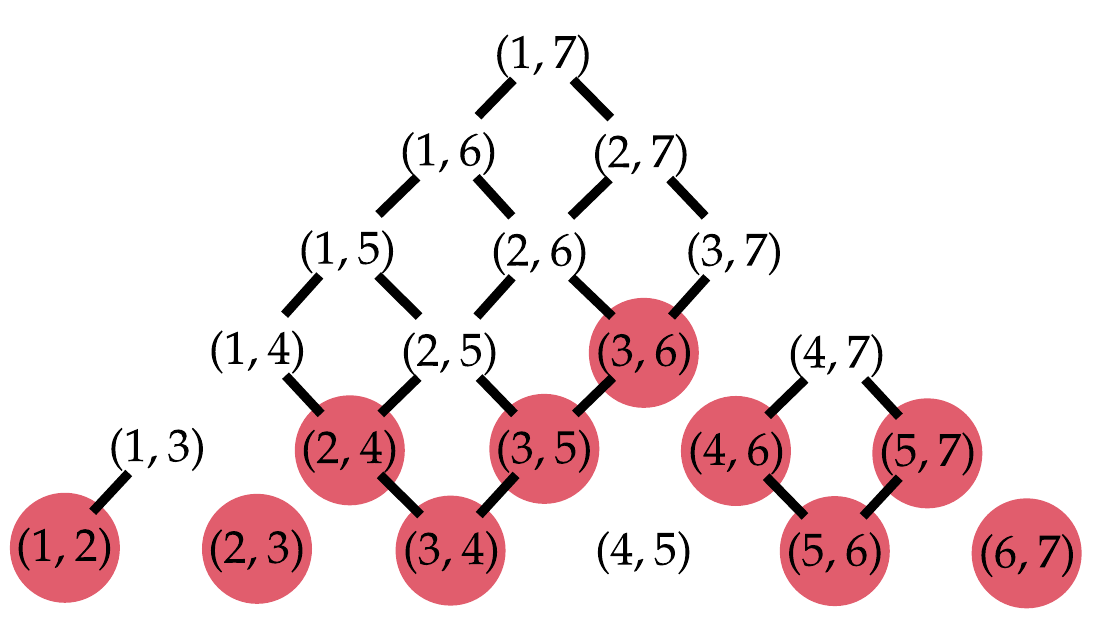} 
\caption{From left to right and from top to bottom: the inversion set of the permutation $w = 3714625$, its right-packing, its left-packing and the lower set $I_\Tree(w)$.}
\label{fig:bij1}
\end{figure}

We use the fact that lower sets of rectangular posets encode shuffles of subpermutations to prove that lower sets of $\RPos_\Tree$ are in bijection with permutations. As we can see in~\cref{fig:moposet9}, a shuffle can be read from the path obtained as the boundary between a lower set of $\RPos_{\Tree,k}$ and their complement, with values of $w^{L,k}$ (resp. $w^{R,k}$) corresponding to south-east (resp. north-east) steps.

\begin{thm}
	Let $\Tree\in \Trees_{n-1}$. $I_\Tree$ is a bijection between $\Ss_n$ and lower sets of $\RPos_\Tree$. We define a partial order $\le_\Tree$ on $\Ss_n$ where $v\le_\Tree w$ if and only if $I_\Tree(v)\subset I_\Tree(w)$. Then $(\Ss_n,\le_\Tree)$ is a distributive lattice.
	\label{pr:bij}
\end{thm}
\begin{proof}
$I_\Tree(w)$ is a lower set of $\RPos_\Tree$, consisting of lower sets of its connected components $\RPos_{\Tree,k}$. By Definition~\ref{def:ideal}, each of these lower sets encodes the shuffle of $w^{L,k}$ and $w^{R,k}$. Assume that we can obtain $w^{L,k}$ (resp. $w^{R,k}$) from the restriction of $I_\Tree(w)$ to $\RPos_{\Tree_1}$ (resp. $\RPos_{\Tree_2}$), where $T_1$ (resp. $T_2$) is the left (resp. right) child of $k$ in $T$. By an immediate recursion, we can reconstruct $w$ from $I_\Tree(w)$, so $I_\Tree$ is injective. We now prove that the number of lower sets of $\RPos_\Tree$ is equal to the number of permutations in $\Ss_n$. The poset $\RPos_\Tree$ consists in the disjoint union of $\RPos_{\Tree_L}$, $\RPos_{\Tree_R}$ and a rectangular poset of size $k\times (n-k)$ where $k$ is the root of $T$, and $\Tree_L$ and $\Tree_R$ its respective left and right children which are of size $k$ and $n-k$. By an immediate recursion, $\RPos_\Tree$ has exactly $k!\cdot (n-k)!\cdot \binom nk = n!$ lower sets. Hence, $I_\Tree$ is bijective.
\end{proof}

\begin{figure}
\centering
\includegraphics[width=0.6\linewidth]{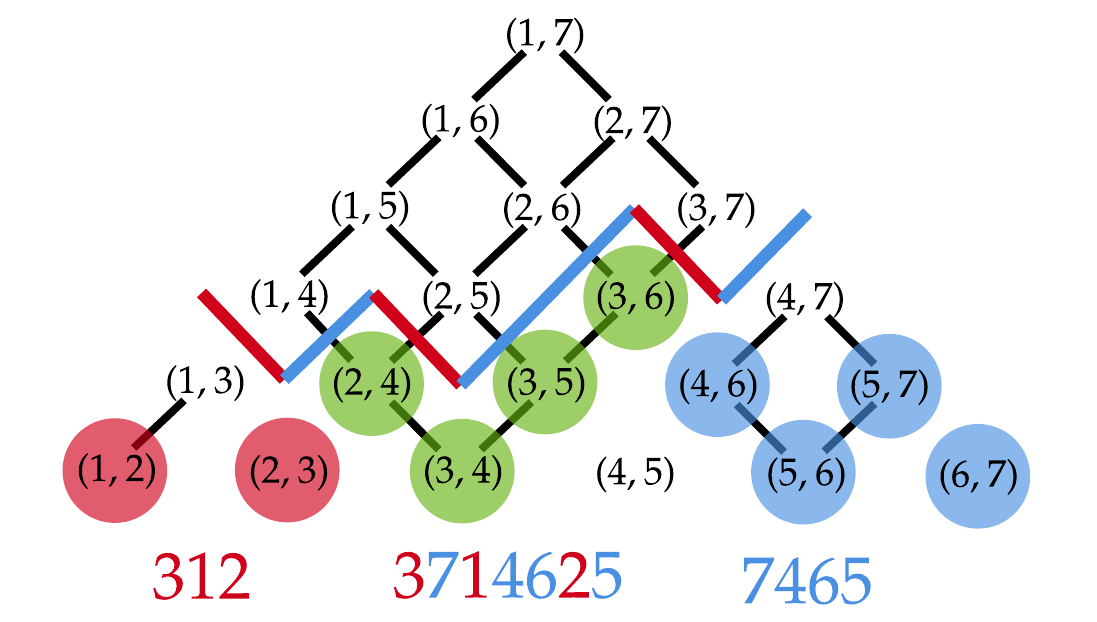}
\caption{The lower set $I_\Tree(w)$ with $w=3714625$ decomposed into lower sets of $\RPos_{\Tree_1}$ (in red), $\RPos_{\Tree_2}$ (in blue), and $\RPos_{\Tree,k}$ (in green). These encode respectively $w^{L,k}=312$, $w^{R,k}=7465$, and the shuffle of the two subpermutations.}
\label{fig:moposet9}
\end{figure}

\begin{rmk}
  The original middle order described by Bouvel, Ferrari, and Tenner~\cite{BFT24} is equal to ${(\Ss_n,\le_\Tree)}$ where $\Tree$ is a left-comb binary tree, \emph{i.e.}, the binary search tree obtained from the permutation $n \! - \! 1 \ n \! - \! 2\dots 2 1$.
\end{rmk}

We now describe cover relations of middle orders on $\Ss_n$.
\begin{prop}
	We have $v\lessdot w$ in $(\Ss_n,\le_\Tree)$ if and only if $w=v\circ (i,j)$ with $\ell(v)<\ell(w)$ and for all $ v^{-1}(i)<a<v^{-1}(j)$, $w(a)\notin [i_{\Tree,k},j_{\Tree,k}]$, where $k$ is such that $(i,j)\in \RPos_{\Tree,k}$.
	\label{pr:covers}
\end{prop}
\begin{proof}
    $\RPos_{\Tree}$ is the disjoint union of $\RPos_{\Tree,k}$ (a rectangular poset of size $k\times (n-k)$), $\RPos_{\Tree_L}$ and $\RPos_{\Tree_R}$, where $k$ is the root of $T$, and $\Tree_L$ and $\Tree_R$ are its left and right children. 
    We have $v\lessdot w$ in $(\Ss_n,\le_\Tree)$ if and only if the corresponding lower sets of $\RPos_{\Tree}$ differ by one element. If they differ by one element of $\RPos_{\Tree,k}$, $v$ and $w$ are shuffles of the same subpermutations $v^{L,k}=w^{L,k}$ and $v^{R,k}=w^{R,k}$, differing by the exchange of two adjacent values $i$ and $j$ with $1\le i\le k<j\le n$. This is indeed equivalent to the fact that for all $v^{-1}(i)<a<v^{-1}(j)$, $w(a)\notin [i_{\Tree,k},j_{\Tree,k}]=[1,n]$. If the lower sets $I_T(v)$ and $I_T(w)$ differ by one element in $\RPos_{\Tree_L}$ (resp. $\RPos_{\Tree_R}$), then $v^{L,k}\lessdot w^{L,k}$ in $(\Ss_k,\le_{\Tree_1})$ (resp. $v^{R,k}\lessdot w^{R,k}$ in $(\Ss_{n-k},\le_{\Tree_2})$). By induction on the size of $\RPos_{\Tree}$, we obtain the result.
    % TO DO Need to say that the reassembling creates no interference
\end{proof}

From the description of cover relations, we deduce that middle orders are between the weak and Bruhat orders.
\begin{prop}
    Middle orders $(\Ss_n,\le_\Tree)$ are refinements of the right weak order and coarsenings of the Bruhat order.
\end{prop}
\begin{proof}
   Covers of the right weak order are given by $v\lessdot w$ if and only if $w=v\circ (i,i+1)$ with $\ell(v)<\ell(w)$, and covers of the Bruhat order are given by $v\lessdot w$ if and only if $w=v\circ (i,j)$ with $\ell(v)<\ell(w)$ and for all $ v^{-1}(i)<a<v^{-1}(j)$, $w(a)\notin [i,j]$. Since if $(i,j)\in \RPos_{\Tree,k}$, $i\le i_{\Tree,k}<j_{\Tree,k}\le j$, by Proposition \ref{pr:covers} we have $$v\le_Rw \implies v\le_\Tree w\implies v\le w.$$
\end{proof}

\section{All distributive lattices between the weak and Bruhat orders}
\label{sec:exhaustivity}
We now prove that our middle orders on $\Ss_n$ are the only distributive lattices between the weak and Bruhat orders. For this we consider \defn{strict Bruhat edges}, \emph{i.e.}, covering relations of the Bruhat order that are not in the weak order. We then study how we can add these edges to the weak order to make a distributive lattice.

\subsection{Equivalence classes of edges}
% TO DO def distributivité dans les préliminaires
    We will use this simple property of distributive lattices throughout this section.
\begin{prop}
	In a distributive lattice $\L$, for all $x,y,z\in \L$, if $y$ and $z$ cover $x$, then $y\vee z$ covers $y$ and $z$. Dually, if $x$ covers $y$ and $z$, then $y$ and $z$ cover $y\wedge z$.
	\label{pr:rhombus}
\end{prop}
\begin{proof}
 Assume that there exists $w\in \L$ between $z$ and $y\vee z$. We have $z=z\wedge (y\vee w)=(z\vee y)\wedge (z\vee w)=w$, hence $y\vee z$ covers $z$. The dual proposition is proved similarly.
\end{proof}

\begin{defi}
	Let $1\le i\le a<b\le j\le n$. We define $\overline {(a,i,j,b)}$ as the set of edges $(v,w)$ of $(\Ss_n,\le)$ such that $w=v\circ (a,b)$ and $\{i,...,j\}$ is the biggest interval containing $\{a,...,b\}$ whose values are not between $a$ and $b$ in $v$ (and $w$).
	\label{def:1}
\end{defi}

For example, if $v=142653$ and $w=152643$, the edge $(v,w)$ is in $\overline{(4,3,5,5)}$. We will only consider sets of edges $\overline {(a,i,j,b)}$ where $i\neq 1$ or $j\neq n$, since if $(v,w)\in \overline {(a,1,n,b)}$, no value lies between $a$ and $b$, so $(v,w)$ would be an edge of the weak order $(\Ss_n,\le_R)$.

\begin{nota}
	Edges (\emph{i.e.}, pairs of permutations  covering one another in the Bruhat order) will be written as permutations with parentheses around the transposed values. We will sometimes replace consecutive entries of a permutation or of an edge with symbols $\word_i$.
	For example $\word_1(a\word_2 b)\word_3$ is an edge with transposed values $(a,b)$, between permutations of the form $\word_1a\word_2b\word_3$ and $\word_1b\word_2a\word_3$.
\end{nota}

\begin{lem}
	Let $E$ be the set of edges of a middle order, $1\le i\le a<b\le j\le n$, and $(v_1,w_1),(v_2,w_2)\in \overline {(a,i,j,b)}$ such that for all $i\le k\le j$, $k$ is on the same side of $a$ and $b$ in $v_1$ and $v_2$, \emph{i.e.}, $v_1^{-1}(k)< v_1^{-1}(a)< v_1^{-1}(b)$ if and only if $v_2^{-1}(k)< v_2^{-1}(a)< v_2^{-1}(b)$. Then $(v_1,w_1)\in E$ if and only if $(v_2,w_2)\in E$.
	\label{lem:1}
\end{lem}
\begin{proof}
	Let $1\le k<\ell\le n$. A permutation of the form $\word_1 k\ell\word_2  a\word_3 b\word_4$ is covered by $\word_1 k\ell\word_2  b\word_3 a\word_4$ and $\word_1 \ell k \word_2  a\word_3 b\word_4$ in the Bruhat order, whose only common upper cover is $\word_1 \ell k \word_2  b\word_3 a\word_4$. Hence, if $\word_1 k\ell \word_2  (a\word_3 b)\word_4\in E$, then $\word_1 \ell k \word_2  (a\word_3 b)\word_4\in E$ by Proposition \ref{pr:rhombus}. The converse is also true, and with the same reasoning we can prove that $$\word_1 (a\word_2k\ell \word_3 b)\word_4\in E\iff \word_1 (a\word_2\ell k \word_3 b)\word_4\in E$$
	and 
	$$\word_1 (a\word_2 b)\word_3 k\ell \word_4\in E\iff \word_1 (a\word_2 b)\word_3 \ell k \word_4\in E.$$
	We now show that if $k\notin \{i,...,j\}$, the position of $k$ with respect to $a$ and $b$ does not change its equivalence class. If $k<a$, we have $k<i-1<a<b$  and $i-1$ is between $a$ and $b$ in all edges of $\overline {(a,i,j,b)}$. A permutation of the form $\word_1 ka\word_2 b\word_3$ is covered by $\word_1 kb\word_2 a\word_3$ and $\word_1 ak\word_2 b\word_3$, whose only common upper cover is $\word_1 bk\word_2 a\word_3$ ($\word_1 ab\word_2 k\word_3$ is not an upper cover because $i-1$ lies between the transposed values). Hence, if $\word_1 k(a\word_2 b)\word_3\in E$, then $\word_1 (ak\word_2 b)\word_3\in E$. We can prove the converse without the assumption that $k\notin \{i,\dots, j\}$. Similarly, we prove that \[\word_1 (a\word_2 b)k\word_3\in E\iff \word_1 (a\word_2 kb)\word_3\in E. \qedhere\]
\end{proof}
We say that edges $(v_1,w_1)$ and $(v_2,w_2)$ are \defn{equivalent} when $(v_1,w_1)\in E$ if and only if $(v_2,w_2)\in E$. We say that $(v_1,w_1)$ \defn{implies} $(v_2,w_2)$  when $(v_1,w_1)\in E$ implies $(v_2,w_2)\in E$.
We now prove that the side of the transposition $(a,b)$ on which each value between $i$ and $j$ lies does not matter in the description of equivalence classes of edges.
\begin{prop}
	Let $E$ be the set of edges of a middle order, $1\le i\le a<b\le j\le n$, and $(v_1,w_1),(v_2,w_2)\in \overline {(a,i,j,b)}$. We have $(v_1,w_1)\in E$ if and only if $(v_2,w_2)\in E$. 
	\label{prop:1}
\end{prop}
\begin{proof}
	Let $(v,w)\in \overline {(a,i,j,b)}$ and $k\notin \{i,...,j\}$ such that $k$ lies on the left side of the transposition in $v$ and $w$. By Lemma \ref{lem:1}, the edge $(v,w)$ is equivalent to an edge of the form $\word_1k(a~i\minusone~j\plusone~b)\word_2$, $\word_1k(a~i\minusone~b)\word_2$ or $\word_1k(a~j\plusone~b)\word_2$. To prove that these edges are respectively equivalent to $\word_1(a~i\minusone~j\plusone~b)k\word_2$, $\word_1(a \, i \minusone \, b)k\word_2$ and $\word_1(a~j\plusone~b)k\word_2$, it is enough to prove the result for permutations of length $5$. We translated logical relations between edges into conjunctive normal form (CNF), and used the SAT solver Glucose~\cite{AuSi18} to prove there exists no middle order on $\Ss_5$ different than those we constructed (see code in Appendix \ref{app:1}). Since edges of these middle orders satisfy the equivalence relations described above, the result follows.
\end{proof}

\begin{figure}
	\centering
	\includegraphics[width=0.5\linewidth]{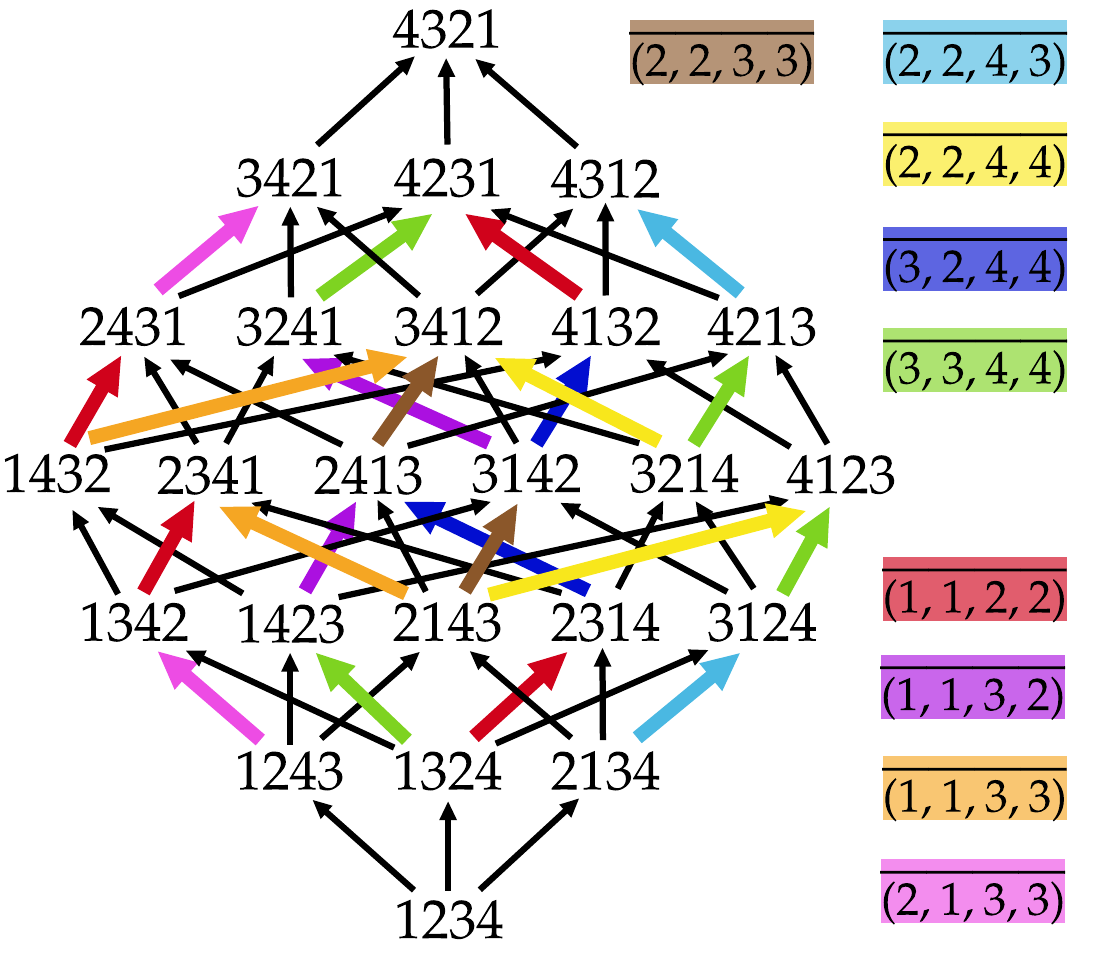}
	\caption{Equivalence classes of strict Bruhat edges in $\Ss_4$.}
	\label{fig:equivalenceclasses}
\end{figure}

\begin{prop}
   \label{pr:exclusion}
	Let $E$ be the set of edges of a middle order. For all $1\le i<j<k\le n$, we have
	\begin{equation*}
		\left[\overline{(i,1,k\minusone,j)}\subset E\right]\iff \left[\overline{(j,i\plusone,n,k)}\not\subset E\right],
	\end{equation*}
	for which we write $\overline{(i,1,k\minusone,j)}\oplus \overline{(j,i\plusone,n,k)}$.
\end{prop}

\begin{proof}
% TO DO citer proposition 4.5
The permutation $\word_1 \, i \, j \, k \, \word_2$ is covered by $\word_1 \, j \, i \, k \, \word_2$ and $\word_1 \, i \, k \, j \, \word_2$, which have exactly two common upper covers $\word_1 \, j \, k \, i \, \word_2$ and $\word_1 \, k \, i \, j \, \word_2$ in the Bruhat order. Hence, we have either $\word_1 \, (j \, i \, k) \, \word_2\in E$ or  $\word_1 \, (i \, k \, j) \, \word_2\in E$ by Proposition~\ref{pr:rhombus}. This relation translates on equivalence classes of edges by Proposition~\ref{prop:1}.
\end{proof}

We now study the implication relations between equivalence classes of edges.
\begin{prop}
	Let $E$ be the set of edges of a middle order, $1\le i_1\le a_1<b_1\le j_1\le n$, $1\le i_2\le a_2<b_2\le j_2\le n$, $(v_1,w_1)\in \overline {(a_1,i_1,j_1,b_1)}$ and $(v_2,w_2)\in \overline {(a_2,i_2,j_2,b_2)}$. If $a_1\le a_2$, $i_1\ge i_2$, $j_1\le j_2$ and $b_1\ge b_2$, then we have
	$$\left[\overline {(a_1,i_1,j_1,b_1)}\subset E\right]\implies \left[\overline {(a_2,i_2,j_2,b_2)}\subset E\right].$$
	\label{pr:implications}
\end{prop}
\begin{proof}
    We first prove that edges of $\overline {(a,i,j,b)}$ imply edges of $\overline {(a,i\minusone,j,b)}$ (if $i>1$) and $\overline {(a,i,j\plusone,b)}$ (if $j<n$). A permutation of the form $\word_1\, b \, i\minusone \, \word_2 \, a \, \word_3$ covers $\word_1\, i\minusone\, b \,\word_2\, a \, \word_3$ and $\word_1\, a \, i\minusone \, \word_2\, b \, \word_3$, whose only common lower cover is $\word_1\, i-1 \, a \,\word_2\, b \, \word_3$.  This shows that $\word_1\, (a \, i-1 \, \word_2 \, b) \, \word_3$ implies $\word_1\,  i-1\,(a  \, \word_2 \, b) \, \word_3$, hence if $\word_2$ contains $i-2$, or if $i=2$, edges of $\overline {(a,i,j,b)}$ imply edges of $\overline {(a,i-1,j,b)}$. By symmetry, edges of $\overline {(a,i,j,b)}$ imply edges of $\overline {(a,i,j+1,b)}$.
    % TO DO expliquer le "similar argument"
    
    We now prove that if $a+1<b$, edges of $\overline {(a,i,j,b)}$ imply edges of $\overline {(a\plusone,i,j,b)}$ and $\overline {(a,i,j,b\minusone)}$. If $i>1$, $j<n$ and $\word_2$ contains $i-1$ and $j+1$, the permutation $\word_1\, a\plusone\, b \,\word_2 \, a \, \word_3$ covers $\word_1\, a\plusone\, a \,\word_2 \, b \, \word_3$ and $\word_1\, a\, b \,\word_2 \, a\plusone \, \word_3$ whose only common lower cover is $\word_1\, a\, a\plusone \,\word_2 \, b \, \word_3$, hence in this case $\word_1\, a\plusone\, (a \,\word_2 \, b) \, \word_3$ implies $\word_1\, a\, (a\plusone \,\word_2 \, b) \, \word_3$.
    If $j=n$, by Proposition \ref{pr:exclusion} we have $\overline{(a,i,n,b)}\oplus \overline{(i\minusone,1,b\minusone,a)}$ and $\overline{(a,i,n,b\minusone)}\oplus \overline{(i\minusone,1,b \! - \! 2,a)}$, and we proved that $\overline{(i\minusone,1,b \! - \! 2,a)}$ implies $\overline{(i\minusone,1,b\minusone,a)}$, hence $\overline{(a,i,n,b)}$ implies $\overline{(a,i,n,b\minusone)}$.
    
    The permutation $\word_1\, a\plusone \, b \, i\minusone \, a \, \word_2$ covers $\word_1 \, i\minusone\, b \, a\plusone \, a \, \word_2$ and $\word_1\, a\plusone \, a \, i\minusone \, b \, \word_2$ which have no common lower cover, so $\word_1 \,( i\minusone\, b \, a\plusone) \, a \, \word_2\in  \overline {(i\minusone,1,b\minusone,a\plusone)}$ and $\word_1\, a\plusone \, (a \, i\minusone \, b)\, \word_2 \in  \overline {(a,i,n,b)}$ cannot be edges of a middle order simultaneously. By Proposition \ref{pr:exclusion}, we have $\overline {(i\minusone,1,b\minusone,a\plusone)} \oplus \overline {(a\plusone,i,n,b)}$, hence $\overline{(a,i,n,b)}$ implies $\overline{(a\plusone,i,n,b)}$.
    
    By symmetry, $\overline{(a,1,j,b)}$ implies $\overline{(a\plusone,1,j,b)}$ and $\overline{(a,1,j,b)}$ implies $\overline{(a,1,j,b\minusone)}$.
    % TO DO itemize les différents cas
\end{proof}

We distinguish two types of strict Bruhat edges. In $\Ss_n$, an edge $(v,w)\in \overline {(a,i,j,b)}$ is said to be a \defn{left edge} if $i=1$. It is said to be a \defn{right edge} if $j=n$.

\begin{figure}[H]
	\centering
	\includegraphics[width=0.9\linewidth]{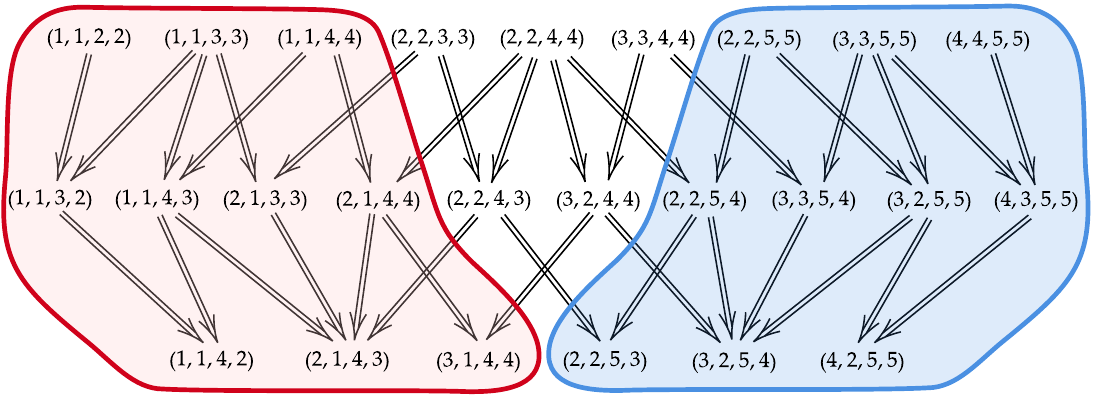}
	\caption{\centering Implications between equivalence classes of strict Bruhat edges in $\Ss_5$, with the left and right edges highlighted.}
	\label{fig:implications5}
\end{figure}

We now show that a middle order is characterized by its left edges, or equivalently by its right edges.
\begin{prop}
	Let $E$ be the set of edges of a middle order. For all $1 < i\le a<b\le j < n$, we have
	\begin{equation*}
		\left[\overline {(a,i,j\plusone,b)}\subset E\right]\wedge \left[\overline {(a,i\minusone,j,b)}\subset E\right]\implies \left[\overline {(a,i,j,b)}\subset E\right].
	\end{equation*}
	\label{pr:leftright}
\end{prop}
\begin{proof}
	Let $\word_2$ containing the value $i-2$ if $i>2$ and $j+2$ if $j<n-1$. A permutation of the form $\word_1\, i\minusone \, a \, \word_2\, j\plusone\, b \, \word_3$ is covered by $\word_1\, i\minusone \, b \, \word_2\, j\plusone\, a \, \word_3$ and $\word_1\, a \, i\minusone \, \word_2\, j\plusone\, b \, \word_3$, whose common upper covers are $\word_1\, a \, b \, \word_2\, j\plusone\, i\minusone \, \word_3$ and $\word_1\, b \, i\minusone \, \word_2\, j\plusone\, a \, \word_3$. This shows that $\overline {(a,i\minusone,j,b)}$ implies either $\overline {(a,i,j,b)}$ or $\overline {(i\minusone,i\minusone,b\minusone,a)}$. By Proposition \ref{pr:implications}, $\overline {(i\minusone,i\minusone,b\minusone,a)}$ implies $ \overline {(i\minusone,1,b\minusone,a)}$, and $\overline {(a,i,j\plusone,b)}$ implies $\overline {(a,i,n,b)}$. By Proposition \ref{pr:exclusion}, $\overline {(i\minusone,1,b\minusone,a)}$ and $ \overline {(a,i,n,b)}$ cannot be both edges of $E$, hence $\overline {(a,i,j,b)}\subset E$.
\end{proof}
Since by Proposition \ref{pr:implications} we also have $$ \left[\overline {(a,i,j,b)}\subset E\right]\implies \left[\overline {(a,i,j\plusone,b)}\subset E\right]\wedge \left[\overline {(a,i\minusone,j,b)}\subset E\right],$$
the edges of a middle order that are not left or right edges are entirely determined by left and right edges. By Proposition \ref{pr:exclusion}, the set of left edges of a middle order is determined by its right edges (and reciprocally), hence we can focus on left edges only.

\begin{prop}
\label{pr:diamond}
	Let $E$ be the set of edges of a middle order. For all $1\le i<j<k<\ell\le n$, we have
	\begin{align*}
		&\left[\overline{(i,1,k\minusone,j)}\subset E\right]\wedge \left[\overline{(j,1,\ell\minusone,k )}\not\subset E\right]\\
		\iff &  \left[\overline{(i,1,\ell\minusone ,j)}\subset E\right]\wedge \left[\overline{(i,1,\ell \minusone,k)}\not\subset E\right]
	\end{align*}
%		\begin{align*}
%			 &\left[\overline{(i,1,k\minusone,j)}\subset E\right]\wedge \left[\overline{(k,j\plusone,n,\ell ) }\subset E\right]\\
%			\iff &  \left[\overline{(i,1,\ell \minusone ,j)}\subset E\right]\wedge \left[\overline{(k,i\plusone,n,\ell )}\subset E\right]
%		\end{align*}
\end{prop}
\begin{proof}
  A permutation of the form $\word_1 \, i \, k \, j \, \ell \, \word_2$ is covered by $\word_1 \, i \, \ell \, j \, k \, \word_2$ and $\word_1 \, j \, k \, i \, \ell \, \word_2$, whose only common cover is $\word_1 \, j \, \ell \, i \, k \, \word_2$. This shows that if $\overline{(i,1,k\minusone,j)}\subset E$ and $\overline{(k,j\plusone,n,\ell ) }\subset E$, then $\overline{(i,1,\ell\minusone ,j)}\subset E$ and $\overline{(k,i  \plusone ,n,\ell )}\subset E$. The converse is also true by the same argument. By Proposition \ref{pr:exclusion}, this property can be then expressed on left edges only.
\end{proof}

\subsection{Sets of edges and Gelfand-Tsetlin triangles}
\label{sec:GT}
A \defn{Gelfand-Tsetlin pattern} of size $n$ is an array $(X_{i,j})_{1\le j\le i\le n}$ such that for all $1\le j\le i \le n-1$, $X_{i+1,j}\le X_{i,j}\le X_{i+1,j+1}$. Let $\GT_n$ be the lattice of Gelfand-Tsetlin triangles with first line $1,2,...,n$ ordered by coordinatewise comparison. $\GT_n$ is a distributive lattice of size $2^{\binom n 2}$~\cite[Proposition 2.1]{CLP02}.
\begin{prop}
	The lattice of sets of left edges satisfying the equivalence and implication relations ordered by inclusion is isomorphic to $\GT_{n-1}$.
	\label{pr:3}
\end{prop}
\begin{proof}
    % TO DO dire quelles inégalités correspond à quoi
	Let $E$ be a set of left edges satisfying the equivalence and implication relations. For all $1\le j\le i\le n-1$, let $X_{i,j}=k$ where $k$ is the largest value such that $\overline{ (j,1,n \! - \! i\! + \! j \! - \! 1,k)}\subset E$. Using Proposition \ref{pr:implications}, it is straightforward to check that $(X_{i,j})_{1\le j\le i\le n-1}$ is in $\GT_{n-1}$. Conversely, if $(X_{i,j})_{1\le j\le i\le n-1}\in \GT_{n-1}$, we can construct a set of left edges satisfying the equivalence and implication relations by letting $\overline{ (j,1,n \! -\! i \!+ \! j \! -\! 1,k)}\subset E$ for all $1\le j\le i\le n$ and $j+1\le k\le X_{i,j}$. Sets of left edges contain each other if and only if the corresponding Gelfand-Tsetlin triangles are comparable coordinatewise, as $X_{i+1,j}\le X_{i,j}$ if and only if $$\left[\overline {(j,1,n-i+j-2,k)}\subset E\right]\implies \left[\overline {(j,1,n-i+j-1,k)}\subset E\right],$$
	and $X_{i,j}\le X_{i+1,j+1}$ if and only if $$\left[\overline {(j,1,n-i+j-1,k)}\subset E\right]\implies \left[\overline {(j+1,1,n-i+j-1,k)}\subset E\right].$$
\end{proof}

By Birkhoff's representation theorem (see~\cref{sec:distri}), Proposition \ref{pr:3} is equivalent to the fact that the poset of equivalence classes of left edges of $\Ss_n$ ordered by implication relations is isomorphic to the poset of join-irreducible elements of $\GT_{n-1}$.

\begin{figure}[H]
	\centering
	\includegraphics[width=0.7\linewidth]{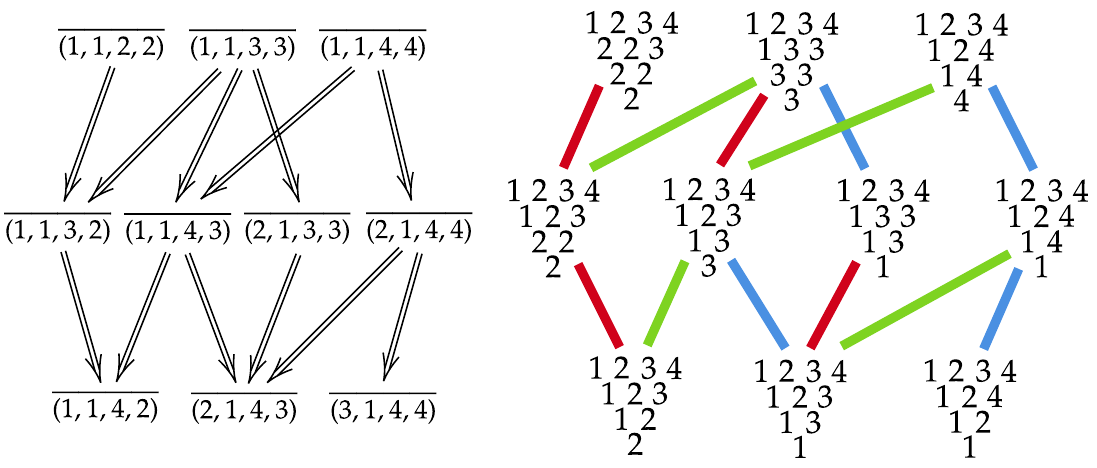}
	\caption{The poset of equivalence classes of left edges in $\Ss_5$ ordered by implications and the poset of irreducible elements of $\GT_4$.}
	\label{fig:gtposet0}
\end{figure}

\begin{prop}
% TO DO mettre des refs vers les propositions
	A set of left edges of $\Ss_n$ satisfying equivalence relations (\emph{cf.} Proposition \ref{prop:1}) and implication relations (\emph{cf.} Proposition \ref{pr:implications}) satisfies the equation of Proposition \ref{pr:diamond}
		\begin{align*}
			&\left[\overline{(i,1,k\minusone,j)}\subset E\right]\wedge \left[\overline{(j,1,\ell\minusone,k )}\not\subset E\right]\\
			\iff &  \left[\overline{(i,1,\ell\minusone ,j)}\subset E\right]\wedge \left[\overline{(i,1,\ell \minusone,k)}\not\subset E\right]
		\end{align*}
	 if and only if the corresponding Gelfand-Tsetlin triangle $(X_{i,j})_{1\le i\le j\le n-1}$ satisfies
	 \begin{equation*}
	 	X_{i,j}=k\implies X_{k+i-j,k}=X_{j+n-k,j}=k,
	 	\label{eq:4}
	 \end{equation*}
	 These Gelfand-Tsetlin triangles are in bijection with binary trees with $n$ leaves (see Figure \ref{fig:gtposet2}).
	\label{pr:5}
\end{prop}
\begin{proof}
   By Proposition \ref{pr:3}, the first equation translates into $$j\le X_{i-\ell+n,i}<k \iff j\le X_{i-k+n,i} \ \text{ and  } \ X_{j-\ell+n,j}<k,$$
   which is equivalent to the second equation. This means that for all $1\le k\le n-1$, the entries such that $X_{i,j}=k$ form a parallelogram, with two opposite vertices given the entry on the top row $X_{n-1,k}=k$ and the bottommost equal to $k$. These triangles are in bijection with binary trees, as a binary tree can be obtained from $(X_{i,j})_{1\le i\le j\le n-1}$ by drawing edges separating adjacent entries with distinct values.
   % TODO éclaircir parallélogramme (une valeur sur la ligne du haut ...) + bijection
\end{proof}

\begin{rmk}
   The order induced on binary trees by the coordinatewise comparison of the corresponding Gelfand-Tsetlin triangles is the Tamari lattice. This corresponds to middle orders ordered by the inclusion of left edges.
\end{rmk}

\begin{figure}[H]
	\centering
	\includegraphics[width=0.7\linewidth]{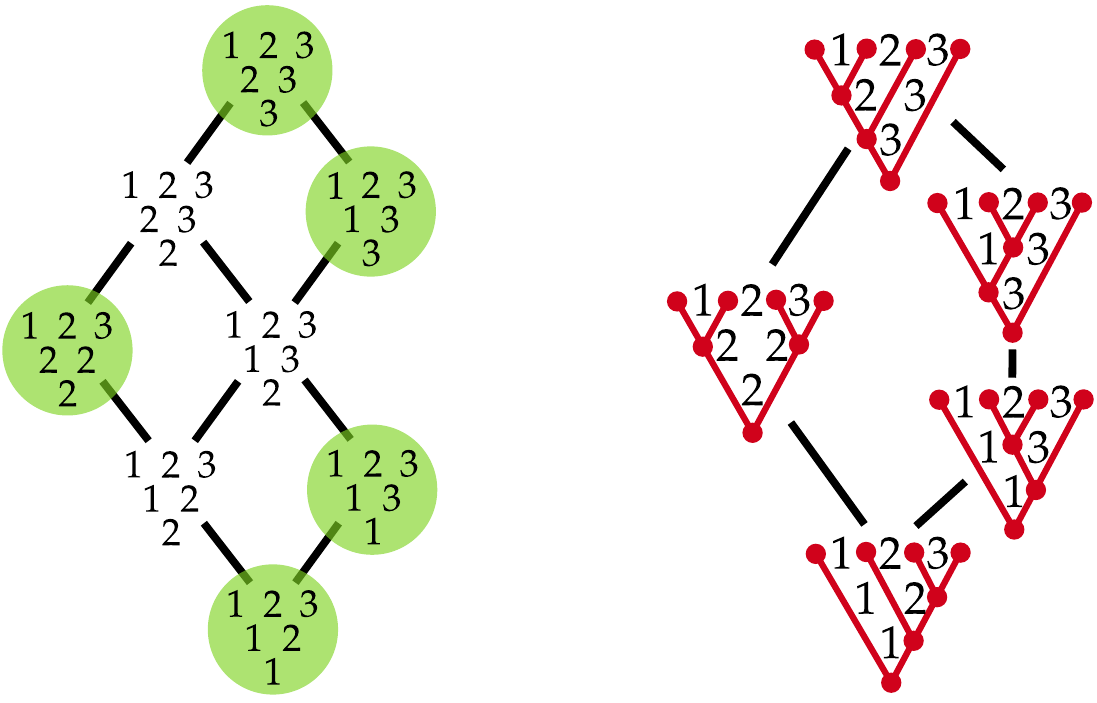}
	\caption{\centering The lattice $\GT_{n-1}$ with $n=4$ with its elements satisfying Equation \eqref{eq:4} being highlighted, and the corresponding binary trees.}
	\label{fig:gtposet2}
\end{figure}

\begin{thm}
    Every distributive lattice between the weak and Bruhat orders on $\Ss_n$ is equal to $(\Ss_n,\le_\Tree)$ for some $\Tree \in \Trees_n$.
\end{thm}
\begin{proof}
% TO DO être plus précis (quelles prop.)
 Throughout this section, we have described properties that sets of edges must satisfy in order to form a middle order (\emph{cf.} Propositions~\ref{prop:1}, \ref{pr:exclusion}, \ref{pr:implications}, \ref{pr:leftright}, and \ref{pr:diamond}), and shown that sets of edges satisfying these properties are in bijection with binary trees (\emph{cf.} \ref{pr:3}), so there is at most $C_{n-1}$ middle orders on $\Ss_n$. In Section \ref{sec:construction} we constructed $C_{n-1}$ middle orders on $\Ss_n$, hence we got all middle orders on $\Ss_n$.
\end{proof}

\section{Properties of middle orders}
\label{sec:properties}
\subsection{Middle orders up to isomorphism}
We have constructed $C_{n-1}$ middle orders on $\Ss_n$, although some of them are isomorphic. Distributive lattices are isomorphic if and only if their posets of irreducibles are isomorphic, so we study which trees $\Tree$ yield isomorphic posets $\RPos_{\Tree}$. 

\begin{prop}
	$\RPos_{\Tree_1}$ and $\RPos_{\Tree_2}$ are isomorphic if and only if $\Tree_1$ can be obtained from $\Tree_2$ by a sequence of exchanges of two subtrees which have the same size or are the left and right children of the same node.
	\label{pr:2}
\end{prop}
\begin{proof}
	If $\Tree_1$ and $\Tree_2$ differ by the exchange of two subtrees of the same size or of the left and right children of a node, $\RPos_{\Tree_1}$ and $\RPos_{\Tree_2}$ are isomorphic. Conversely, if $\RPos_{\Tree_1}$ and $\RPos_{\Tree_2}$ are isomorphic, we reconstruct $\RPos_{\Tree_1}$ using the rectangular posets from $\RPos_{\Tree_2}$, from the largest to the smallest (with respect to the sum of lengths of the sides). Each rectangle needs to be moved and/or flipped (along with all rectangles in the corresponding subtree), which corresponds to exchanging subtrees of the same size, or exchanging left and right children of a node.
	% TO DO moved and/or flipped with the whole subtree
\end{proof}

In particular, if $\Tree_1$ and $\Tree_2$ are isomorphic binary trees (up to exchange of left and right subtrees for each node), the posets $\RPos_{\Tree_1}$ and $\RPos_{\Tree_2}$ are isomorphic. The converse is not true (see Figure \ref*{fig:isomo1}).

\begin{figure}[H]
	\centering
	\includegraphics[width=0.9\linewidth]{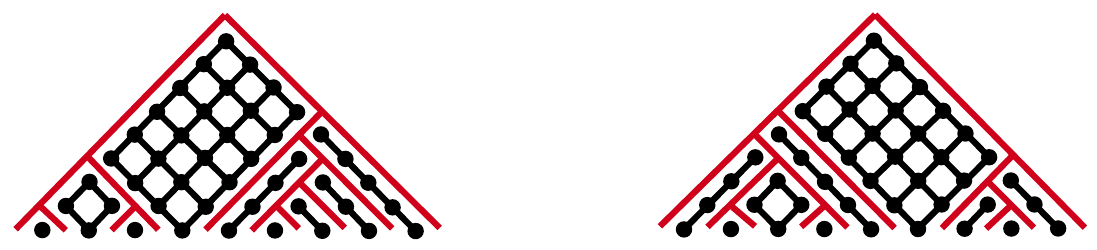}
	\caption{Isomorphic posets corresponding to non-isomorphic binary trees.}
	\label{fig:isomo1}
\end{figure}

The subtrees sizes of a binary tree $\Tree$ are the multiset containing the number of internal nodes of each subtree of $\Tree$, or equivalently $m+n-1$ for each rectangular poset of dimensions $m\times n$ in $\RPos_\Tree$. If $\RPos_{\Tree_1}$ and $\RPos_{\Tree_2}$ are isomorphic, then the binary trees $\Tree_1$ and $\Tree_2$ have the same subtree sizes. The converse is not true (see Figure \ref{fig:isomo2}).

% TO DO obtenu avec les $m+n-1$

\begin{figure}[H]
	\centering
	\includegraphics[width=0.9\linewidth]{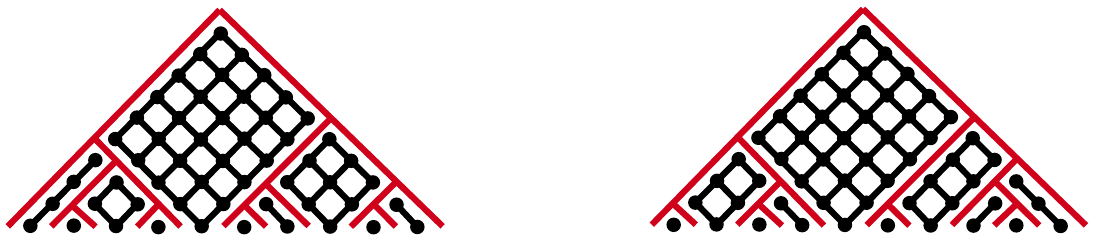}
	\caption{\centering Non-isomorphic posets with their subtree sizes both equal to $(10,5,4,3,2,2,1,1,1,1)$.}
	\label{fig:isomo2}
\end{figure}

These observations give us lower and upper bounds for the number of middle orders up to isomorphism~\cite[A247139]{oeis}, as it is smaller than the number of binary trees up to isomorphism~\cite[A001190]{oeis}, and larger than the number of binary trees up to their subtrees sizes~\cite[A382440]{oeis}. Similarly we can consider binary trees up to their ordered coordinates in the Loday polytopal realization of the Tamari lattice. These can be obtained as the multisets containing $mn$ for each rectangle of dimensions $m\times n$ (instead of $m+n-1$) in $\RPos_\Tree$. Enumerating distinct multisets of Loday coordinates yields a better lower bound for the number of middle orders up to isomorphism.

\begin{table}[H]
	\centering
	\begin{tabular}{cccccccccccccccc}
	\toprule
	   $n$ & $1$& $2$ &$3$ & $4$ & $5$ & $6$ & $7$ & $8$ & $9$ & $10$ & $11$ & $12$ & $13$ & $14$& $15$ \\ \midrule 
	   $A382440$ & $1$ & $1$ & $1$ & $2$ & $3$ & $6$ & $11$ & $23$ & $45$ & $95$ & $194$ & $414$ & $863$ & $1850$ & $3910$\\ \midrule 
	   $A392457$ & $1$ & $1$ & $1$ & $2$ & $3$ & $6$ & $11$ & $23$ & $45$ & $95$ & $195$ & $417$ & $865$ & $1877$ & $4001$\\ \midrule 
		$\mathbf{A247139}$& $1$ & $1$ & $1$ & $2$ & $3$ & $6$ & $11$ & $23$ & $45$ & $95$ & $195$ & $417$ & $875$ & $1887$ & $4021$\\ \midrule 
		$A001190$ & $1$ & $1$ & $1$ & $2$ & $3$ & $6$ & $11$ & $23$ & $46$ & $98$ & $207$ & $451$ & $983$ & $2179$ & $4850$\\ \bottomrule
	\end{tabular}
	\caption{\centering The number of middle orders on $\Ss_n$ up to isomorphism (A247139), compared to the number of binary trees up to isomorphism (A001190), the number of binary trees up to their subtrees sizes (A382440), and the number of Loday coordinates up to permutation (A392457), in increasing order of term values.}
\end{table}

\subsection{Chains and intervals}
\label{sec:chains}
We give formulas for the number of chains of a given length (and in particular intervals) in middle orders, and we show it depends only on the length $n$ of the permutations. We give combinatorial proofs of these formulas, which are a particular case of Proposition \ref{pr:Fpoly}.
A chain of a finite lattice $\L$ is said to be a \defn{complete chain} if it contains $\bigvee \L$ and $\bigwedge \L$. We take the convention that a chain $x_0,...,x_\ell$ has length $\ell$.
Complete chains of a finite distributive lattice $\IdL(\Pos)$ of length $\ell$ are in bijection with surjective increasing functions $f:\Pos\to \{1,...,\ell\}$.
% TODO trouver une ref

	Increasing functions from a rectangular poset of size $n\times m$ to $\{1,...,\ell\}$ are in bijection with plane partitions contained in a box of dimensions $m\times n\times (\ell-1)$, hence their number is equal to $\Pi(n,m,\ell-1)$, where~\cite{Macd95}
$$\Pi(n,m,\ell):=\prod_{i=1}^n\prod_{j=1}^m\prod_{k=1}^\ell\frac{i+j+k-1}{i+j+k-2}.$$
% TO DO citer MacMahon + <===> plane partitions.

\begin{prop}
   For all $\ell,n\ge 1$, all middle orders of size $n$ have the same number of complete chains of length $\ell$, which is equal to 
	$$\sum_{k=0}^{\ell}(-1)^{\ell-k}\binom \ell k\prod_{i=0}^{n-1}\binom {k+i-1}{i}.$$
	\label{pr:1}
\end{prop}
\begin{proof}
	We prove by induction that for all binary trees $T$ with $n-1$ internal nodes, the number of increasing functions $f:\RPos_{\Tree}\to\{1,...,\ell\}$ is $$\displaystyle \prod_{i=0}^{n-1}\binom {\ell+i-1}{i}.$$

	If $n = 1$ the result is trivial. If $n>1$, the poset of irreducibles of a middle order of size $n$ can be written $\RPos_{\Tree}=\RPos_{\Tree,k}\sqcup \RPos_{\Tree_1}\sqcup \RPos_{\Tree_2}$ where $\RPos_{\Tree,k}$ is a rectangular poset of size $n_1\times n_2$, and $\Tree_1$ and $\Tree_2$ are the left and right children of the root of $\Tree$ and have respective sizes $n_1$ and $n_2$ with $n=n_1+n_2$. Hence, the number of increasing functions from $\RPos_{\Tree}$ to $\{1,...,\ell\}$ is
		$$\Pi(n_1,n_2,\ell-1)\cdot \prod_{i=0}^{n_1-1}\binom {\ell+i-1} i \cdot \prod_{i=0}^{n_2-1}\binom {\ell+i-1} i\\
		=\prod_{i=0}^{n-1}\binom {\ell+i-1} i.$$
		We then use the inclusion-exclusion principle to compute the number of such functions that are surjective.
	
\end{proof}

% TODO expliquer le calcul
In particular the number of intervals of any middle order on $\Ss_n$ is equal to $$\prod_{i=0}^{n-1}\binom{i+2}i=\frac {n!(n+1)!}{2^n},$$
as obtained in \cite[Corollary 3.2]{BFT24} for the original middle order.

\begin{defi}[Boolean chain]
	In a finite lattice $\L$, a (complete) \defn{boolean chain} of length $\ell$ is a complete chain $x_0<...<x_\ell\in \L$ such that for all $0\le i<\ell$, the interval $[x_i,x_i+1]$ is boolean.
\end{defi}

Boolean chains of a finite distributive lattice $\IdL(\Pos)$ of length $\ell$ are in bijection with surjective strictly increasing functions $f:\Pos\to \{1,...,\ell\}$.
% TODO ref

\begin{prop}
	For all $\ell\ge 1$, all middle orders of size $n$ have the same number of boolean chains of length $\ell$, which is equal to $$\sum_{k=0}^{\ell}(-1)^{\ell-k}\binom \ell k\prod_{i=0}^{n-1}\binom ki.$$
	\label{pr:bool}
\end{prop}
\begin{proof}
	The number of strictly increasing functions from a rectangular poset of size $n\times m$ to $\{1,...,\ell\}$ is equal to $\Pi(n,m,\ell-m-n+1)$. With the same technique as in the proof of Proposition \ref{pr:1}, we prove that for all binary trees $T$ with $n-1$ internal nodes, the number of strictly increasing functions $f:\RPos_{\Tree}\to\{1,...,\ell\}$ is $$\displaystyle \prod_{i=0}^{n-1}\binom {\ell}{i}.$$
	We then use the inclusion-exclusion principle to obtain the number of surjective functions.
\end{proof}

For all $\sigma\in \Ss_n$, let $\overline \sigma $ be the image of $\sigma$ by central symmetry. The map $\sigma \mapsto \overline{\sigma}$ is an antiautomorphism of all middle orders on $\Ss_n$. We say that a subset of $\Ss_n$ is self-dual if it is invariant by this map. The number of self-complementary lower sets of a product poset $n\times m\times \ell$ is equal to $\Pi(n,m,\ell)_{q=-1}$ \cite{Stem94}, where
$$\Pi(n,m,\ell)_q:=\prod_{i=1}^n\prod_{j=1}^m\prod_{k=1}^\ell\frac{[i+j+k-1]_q}{[i+j+k-2]_q}.$$
% TODO partition planes autoconjuguées + q = graduation, cyclic sieving q = -1

We give the following propositions without proof, as they use the same arguments as Propositions \ref{pr:1} and \ref{pr:bool}, except that the inclusion-exclusion is done on self-dual sets.
\begin{prop}
	For all $\ell,n\ge 1$, all middle orders of size $n$ have the same number of self-dual complete chains of length $\ell$, which is equal to 
	$$\sum_{k=0}^{\ell}(-1)^{\lceil(\ell-k)/2\rceil}\begin{pmatrix}\ell \\ k\end{pmatrix}_{q=-1}\prod_{i=0}^{n-1}\begin{pmatrix}k+i-1\\ i\end{pmatrix}_{q=-1}$$ where $$
		\begin{pmatrix}n\\ k\end{pmatrix}_{q=-1}=\begin{cases}
			0 & \text{if $n$ is even and $k$ is odd;} \\
			\begin{pmatrix}\lfloor n/2\rfloor\\ \lfloor k/2\rfloor\end{pmatrix} & \text{otherwise.}
		\end{cases}$$
\end{prop}
%TODO sketch of proof
\begin{prop}
	For all $\ell,n\ge 1$, all middle orders of size $n$ have the same number of self-dual boolean chains of length $\ell$, which is equal to 
	$$\sum_{k=0}^{\ell}(-1)^{\lceil(\ell-k)/2\rceil}\begin{pmatrix}\ell \\ k\end{pmatrix}_{q=-1}\prod_{i=0}^{n-1}\begin{pmatrix}k\\ i\end{pmatrix}_{q=-1}.$$
\end{prop}

\section{Middle orders on parabolic quotients of finite Coxeter groups}
\label{sec:parabolic}
Let $W$ be a finite Coxeter group with set of generators $S$, and $J\subset S$. Let $W_J$ be the parabolic subgroup generated by $J$. Cosets in $W_J\backslash W$ are intervals of the left weak order (though we will consider later the right weak order on $W_J$), and minimal elements of these cosets form an interval of the right weak order, which we write as $^JW$. Cosets are also intervals of the Bruhat order, and the Bruhat order restricted to minimal elements defines the Bruhat order on $^JW$. The weak order on $^JW$ is refined by the Bruhat order, so we can consider middle orders on $^JW$ as we did for $W$.
Every element $w\in W$ has a unique factorization $w_J\cdot \, ^Jw$ with $^Jw\in \, ^JW$ and $w_J\in W_J$ (see \cite{BjBr05}, Proposition 2.4.4).
Using this, we obtain the following inclusion of orders:

\begin{equation}
    \label{eq:inclusion}
    (W,\le_R)\subset (W_J,\le_R)\times (^JW,\le_R) \subset  (W_J,\le)\times (^JW,\le)\subset (W,\le).
\end{equation}

Hence, if there exists middle orders on $W_J$ and $^JW$, their product defines a middle order on $W$.

\begin{figure}[H]
\centering
\includegraphics[width=0.45\linewidth]{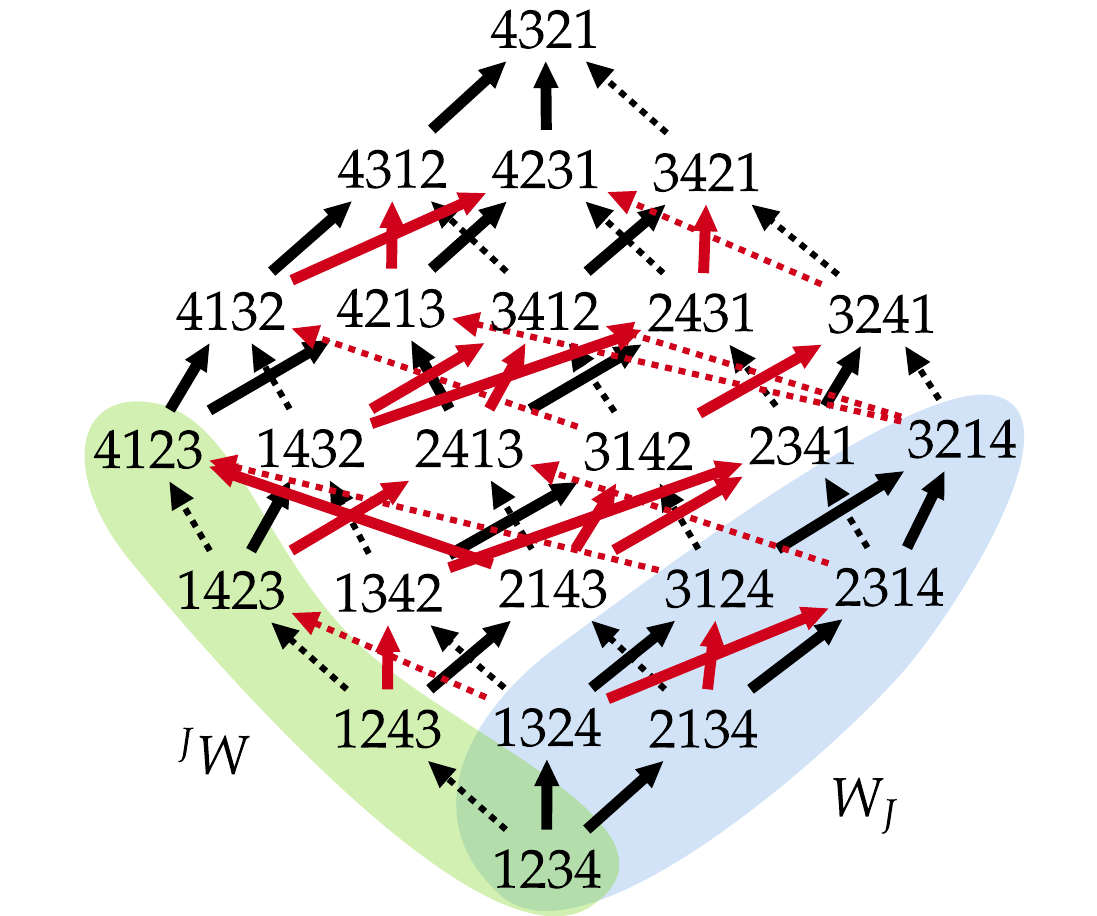}
\includegraphics[width=0.45\linewidth]{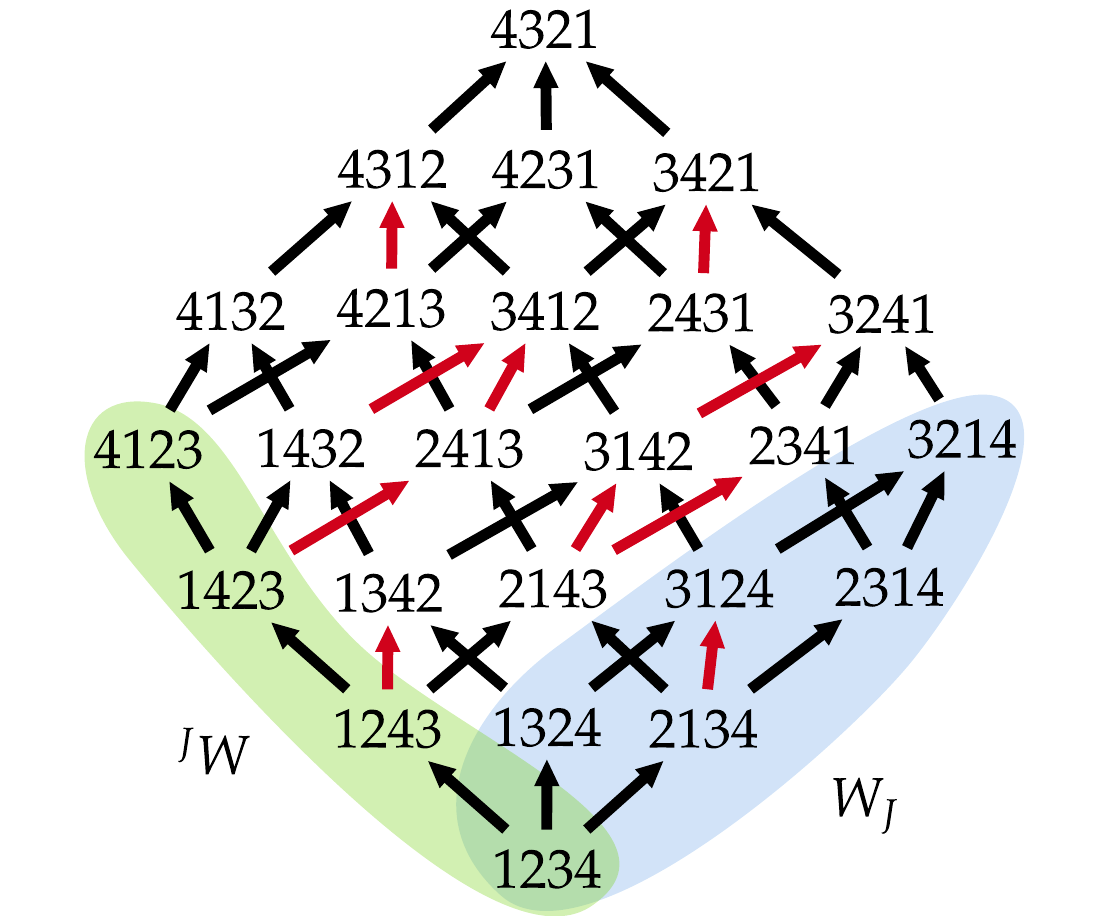}
\caption{\centering On the left: the weak and Bruhat orders on $\Ss_4$, with a parabolic subgroup $W_J$ and its quotient $^JW$ given by $J=\{(12),(23)\}$. On the right: A middle order on $\Ss_4$, obtained as the product of middle orders on $^JW$ and $W_J$.}
\label{fig:quotient}
\end{figure}

Let us focus on the case of the symmetric group. Elements of $W_J\backslash \Ss_n$ can be seen as \defn{multipermutations} with $\alpha_k$ occurrences of $k$ for $1\le k\le n-|J|$, where $\alpha$ is an integer composition of $n$ such that $s_i\notin J$ if and only if we can write $i=\alpha_1+\alpha_2+\dots+\alpha_k$ for some $k$. We define $\Ss_\alpha$ as the set of such multipermutations.

We build middle orders on $\Ss_\alpha$ that are in bijection with binary trees on $n-|J|$ leaves, and prove they are the only middle orders on $W_J\backslash \Ss_n$. The construction of middle orders and proof of their exhaustivity work exactly as what we did in Sections \ref{sec:construction} and \ref{sec:exhaustivity}, so we will only explain it shortly.

Multipermutations are encoded by \defn{partial inversions sets,} \emph{i.e.}, inversion sets of their standardization, \emph{i.e.}, the minimal element of their corresponding coset in $W_J\backslash \Ss_n$. Multipermutations in $\Ss_\alpha\cong W_J\backslash \Ss_n$ with $J\subset\{s_1,...,s_{n-1}\}$ have their inversions sets contained in $\Pos_\alpha$, which we define as the subposet of $\Pos_n$ containing $(i,j)$ if $s_i\notin J$ or $s_{j-1}\notin J$. 
Let $U$ be a binary search tree whose set of labels is $\{1,...,n-1\}\setminus J$. $U$ defines a rectangulation of $\Pos_\alpha$, which we write $\RPos_{U,n}$, obtained by restricting the order of $\RPos_T$ to $\Pos_\alpha$ where $\Tree$ is any completion of $\PTree$ into a binary tree whose set of labels is $\{1,...,n-1\}$. We then define $I_{U,n}(w)$ as the restriction of $I_\Tree(w)$ to $\RPos_{U,n}$. 
Similarly to Theorem~\ref{pr:bij}, this bijection defines a middle order on $\Ss_\alpha$:
\begin{thm}
	Let $U$ be a binary search tree with set of labels $\{1,...,n-1\}\setminus J$. $I_{U,n}$ is a bijection between $\Ss_\alpha$ and lower sets of $\RPos_{U,n}$. We define a partial order $\le_{U,n}$ on $\Ss_\alpha$ where $v \le_{U,n} w$ if and only if $I_{U,n}(v)\subset I_{U,n}(w)$. Then $(\Ss_\alpha,\le_{U,n})$ is a distributive lattice.
\end{thm}

\begin{figure}[H]
\centering
\includegraphics[width=0.9\linewidth]{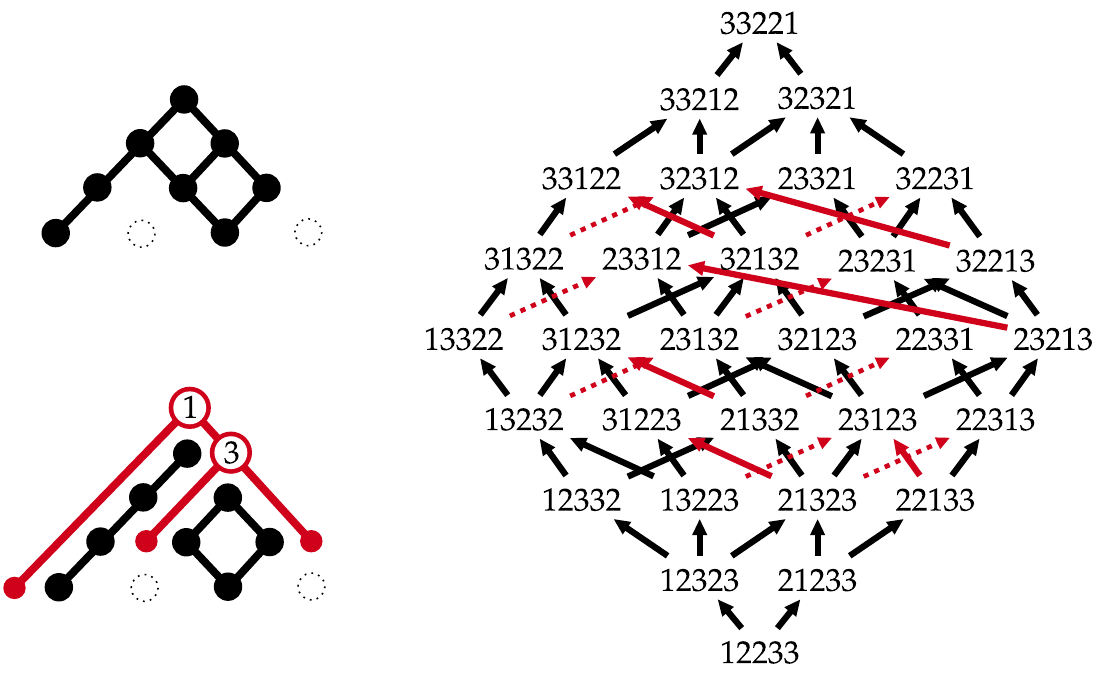}
\caption{On the right: a middle order on $\Ss_{122}$ induced by a binary tree $U$ with set of labels $\{1,3\}$. On the left: the posets $\Pos_{122}$ and $\RPos_{U,5}$.}
\label{fig:quotient3}
\end{figure}

%TODO retravailler
This construction yields all middle orders on $\Ss_\alpha$, because if it was not the case, we could construct middle orders on $\Ss_n$ different from those we already know (we will see in Section \ref{sec:minuscule} that all middle orders on $\Ss_n$ are obtained as a product of middle orders on a parabolic subgroups and on its quotient). We can also prove it by generalizing the results of Section~\ref{sec:exhaustivity} to multipermutations. Equivalence classes are defined as in Definition \ref{def:1}, and Lemma \ref{lem:1} becomes:

\begin{lem}
	Let $E$ be the set of edges of a middle order on $\Ss_\alpha$, $1\le i\le a<b\le j\le n$, and $(v_1,w_1),(v_2,w_2)\in \overline {(a,i,j,b)}$ such that for all $i\le k\le j$, $k$ appears with the same multiplicity on each side of $a$ and $b$ in $v_1$ and $v_2$. Then $(v_1,w_1)\in E$ if and only if $(v_2,w_2)\in E$.
\end{lem}

As in Proposition~\ref{prop:1}, we use the SAT solver Glucose~\cite{AuSi18} to verify there is no other middle order on $\Ss_\alpha$ for all $\alpha$ of length $5$, to show that the $\overline {(a,i,j,b)}$ are equivalence classes of edges of middle orders on $\Ss_\alpha$. Propositions~\ref{pr:exclusion}, \ref{pr:implications}, \ref{pr:leftright} and \ref{pr:diamond} are generalized without additional work to middle orders on multipermutations. Since equivalence classes of edges on $\Ss_\alpha$ correspond to equivalence classes of edges on $\Ss_{n-|J|}$, results of~\cref{sec:GT} can be generalized to middle orders on multipermutations to obtain the following result.

\begin{thm}
    Every distributive lattice between the weak and Bruhat orders on $\Ss_\alpha$ is equal to $(\Ss_\alpha,\le_{U,n})$ for some binary tree $U$ with labels $\{1,...,n\}\setminus J$.
\end{thm}

\section{Minuscule middle orders}
\label{sec:minuscule}
\subsection{Definition of minuscule middle orders}
We have seen in Section~\ref{sec:parabolic} that the product of middle orders on a parabolic subgroup $W_J$ and its quotient $^JW$ forms a middle order on $W$. 
%This can help us to construct middle orders on Coxeter groups by recursion on the number of generators.
For Weyl groups, parabolic quotients on which the Bruhat order is a distributive lattice have been classified \cite{Pro84}. These quotients are given by $J=S\setminus \{s\}$ where $s$ corresponds to a minuscule root $\alpha_s\in \Phi_+$.
Since they are maximal parabolic quotients, the weak and Bruhat orders coincide, giving exactly one middle order on $^JW$. The quotient $^JW$ is isomorphic to $\IdL(\Pos_s)$, where $\Pos_s$ (which is called a minuscule poset) is the upper set of $\alpha_s$ in $\Phi_+$.
Hence, if $\L$ is a distributive lattice such that $(W_J,\le_R)\subset \L\subset (W_J,\le)$, we then have $(W,\le_R)\subset \L\times \, ^JW \subset (W,\le)$.
 We define \defn{minuscule middle orders} as the middle orders obtained using only such quotients.
 %, as the connected components of their irreducibles posets are minuscule posets. 
\begin{defi}
   Let $(W,S)$ be a Weyl group, and let $s_{i_1},\dots, s_{i_r}$ be a permutation of $S$ such that for all $1\le k\le r$, $\alpha_{s_{i_k}}$ is a minuscule root of $W_{\{s_{i_k},\dots,s_{i_r}\}}$. Minuscule middle orders are posets of the form
    $(W,\le_{s_{i_1},\dots, s_{i_r}}):=(W_J,\le_{s_{i_2},\dots,s_{i_r}})\times \, ^JW$ where $J=S\setminus\{s_{i_1}\}$, which are guaranteed to be middle orders by~\cref{eq:inclusion}.
\end{defi}
\begin{figure}[H]
\centering
\includegraphics[width=0.8\linewidth]{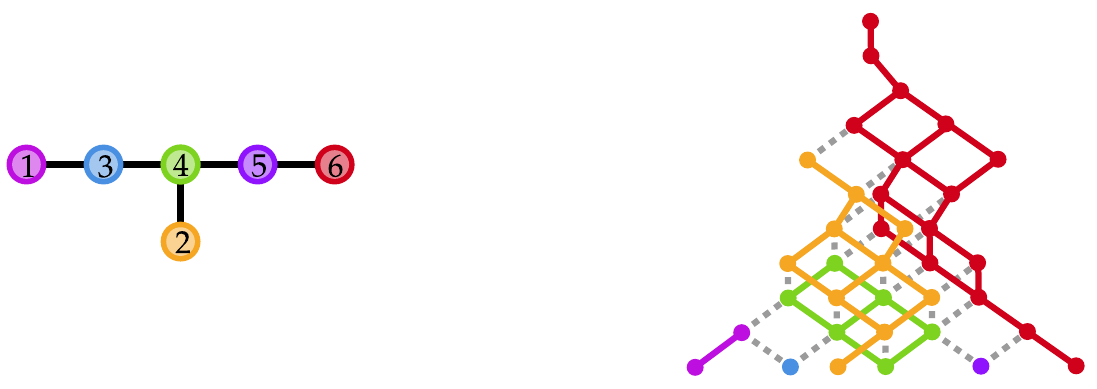}
\caption{\centering The Coxeter diagram of $E_6$, and a partition of its root poset into minuscule posets induced by the permutation $s_6,s_2,s_4,s_1,s_3,s_5$}
\label{fig:quotiente6bis}
\end{figure}

We define $\RPos_{s_{i_1},\dots, s_{i_r}}$ recursively as the disjoint union of $\RPos_{s_{i_2},\dots, s_{i_r}}$ and the minuscule poset $\Pos_{s_{i_1}}$, with $\RPos_\emptyset=\emptyset$. Since a minuscule middle order is the product of distributive lattices whose posets of irreducibles are minuscule posets, we have an isomorphism $$(W,\le_{s_{i_1},\dots, s_{i_r}})\cong \IdL(\RPos_{s_{i_1},\dots, s_{i_r}}).$$
% TODO proposition ?

In type $A$ all middle orders are minuscule, but for other Coxeter groups this is not the case in general. For example, $B_3$ has $4$ middle orders isomorphic to minuscule middle orders on Weyl groups $B_3$ or $C_3$, but it has $8$ middle orders in total (see Figure \ref{fig:typeb}). Weyl groups $E_8$, $F_4$ and $G_2$ have no minuscule middle order since they have no minuscule root.

\begin{figure}[H]
	\centering
	\includegraphics[width=1\linewidth]{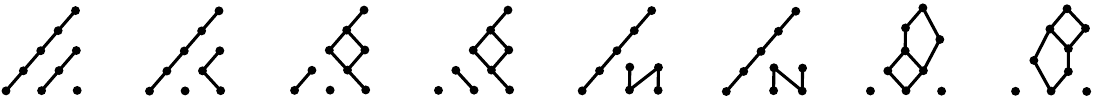}
	\caption[Caption]{\centering Posets of irreducibles of the $8$ middle orders on the Coxeter group $B_3$. The first $4$ are isomorphic to minuscule middle orders on Weyl groups $B_3$ or $C_3$\footnotemark, and the last $2$ are non-minuscule sorting orders (\emph{cf.} Section \ref{sec:sorting}). }
	\label{fig:typeb}
\end{figure}

\footnotetext{Note that only one of the two first is a minuscule middle order on the Weyl group $B_3$ \emph{s.s.}, the other one being obtained by quotienting $B_3$ by $W_J=B_2$ and multiplying by the minuscule middle order on $W_J$ seen as $C_2$. The third and fourth are minuscule middle orders on $C_3$.}

\subsection{Packing in minuscule middle orders}
\label{subsec:minuspacking}
Similarly to the case of the symmetric group (see Proposition~\ref{prop:packing}), the bijection between $W$ and $(W,\le_{s_{i_1},\dots, s_{i_r}})$ and $\IdL(\RPos_{s_{i_1},\dots, s_{i_r}})$ can be seen as packing of inversion sets. 

\begin{defi}
   For all $s\in S$ and $A\subset \Phi_+$, let $\pack_i(A)$ be the subset of $\Phi_+$ obtained from $A$ by removing $\alpha_{s_i}$ if it belongs to $A$, replacing elements $\alpha$ by $\alpha-\alpha_{s_i}$ if $\alpha- \alpha_{s_i}\in \Phi^+\setminus A$, and repeating this process until no such element exists. We write $\pack_{i_1\cdots i_r}(A):=\pack_{i_r}\circ \dots \circ \pack_{i_1}(A)$ where $i_1\dots i_r$ is any word in the generators in $S$.
\end{defi}

See \cref{fig:packingsc4} for an example. By definition $\pack_s(A)$ is the only subset of $\Phi_+$ which does not contain $\alpha_s$ and such that for all $\alpha\in \Phi_+$ with $\alpha\neq \alpha_s$, $A$ and $\pack_s(A)$ contain the same number of roots equal to $\alpha$ modulo $\alpha_s$.

\begin{prop}
\label{prop:packing2}
For all $w\in W$, we have
$$\pack_i(N(w)) = \begin{cases}
N(s_iw) \text{ if $s_i\in S$ is a left descent of $w$ (\emph{i.e.}, $\ell(s_iw)<\ell(w)$);}\\
N(w) \text{ otherwise.}
\end{cases}$$
\end{prop}

\begin{proof}
  The generator $s_i$ is a left descent of $w$ if and only if $\alpha_{s_i}\in N(w)$. Hence, if $s_i$ is not a left descent of $w$, $N(w)$ does not contain any $\alpha$ such that $\alpha-\alpha_{s_i}\in \Phi^+\setminus N(w)$ since $N(w)$ is coclosed, so $N(w)$ is invariant by $\pack_i$. 
  For all $\alpha\in \Phi_+$, if $\alpha\neq \alpha_{s_i}$  we have $s_i(\alpha)\in \Phi_+$ and $s_i(\alpha)=\alpha+k\alpha_{s_i}$ where $k\in \Z$, and $s_i(\alpha_{s_i})=-\alpha_{s_i}$. Hence, if $s_i$ is a left descent of $w$, $N(s_iw)$ does not contain $\alpha_{s_i}$ and for all $\alpha\in \Phi_+$ with $\alpha\neq \alpha_{s_i}$, $N(w)$ and $N(s_iw)$ contain the same number of roots equal to $\alpha$ modulo $\alpha_{s_i}$. $N(s_iw)$ does not contain any $\alpha$ such that $\alpha-\alpha_{s_i}\in \Phi^+\setminus N(s_iw)$, so $N(s_iw)=\pack_i(N(w))$.
\end{proof}
It follows from Proposition~\ref{prop:packing2} that for all $w\in W$, any reduced word $i_1\dots i_r$ for $w$ gives the same value of $\pack_{i_1\dots i_r}(N(w))$. More generally, if $i_1\dots i_r$ is a word (not necessarily reduced) and $j_1\dots j_t$ is a reduced word for $w$, then $\pack_{i_1\dots i_r}(N(w))$ is the element of $W$ corresponding to the evaluation of $i_r\dots i_1 j_1\dots j_t$ in the Coxeter monoid (\emph{i.e.}, the monoid obtained by letting $s^2=s$ instead of $s^2=1$ for all $s\in S$, see~\cite{Tsar90} for more details). 
\begin{cor}
       Let $J\subset S$ and $i_1\cdots i_r$ a reduced word for the maximal element of $W_J$. For all $w\in W$, we have $N(\, ^Jw)=\pack_{i_1\cdots i_r}(N(w))$, where $^Jw\in \, ^JW$ is the minimal element of $W_Jw$.
\end{cor}
\begin{proof}
   To compute the minimal coset representative $^Jw\in \, ^JW$ corresponding to $w\in W$, we successively multiply $w$ on the left by $s_{i_k}$ for all $1\le k \le r$ whenever it reduces its length.
\end{proof}

\begin{figure}
\centering
\includegraphics[width=0.9\linewidth]{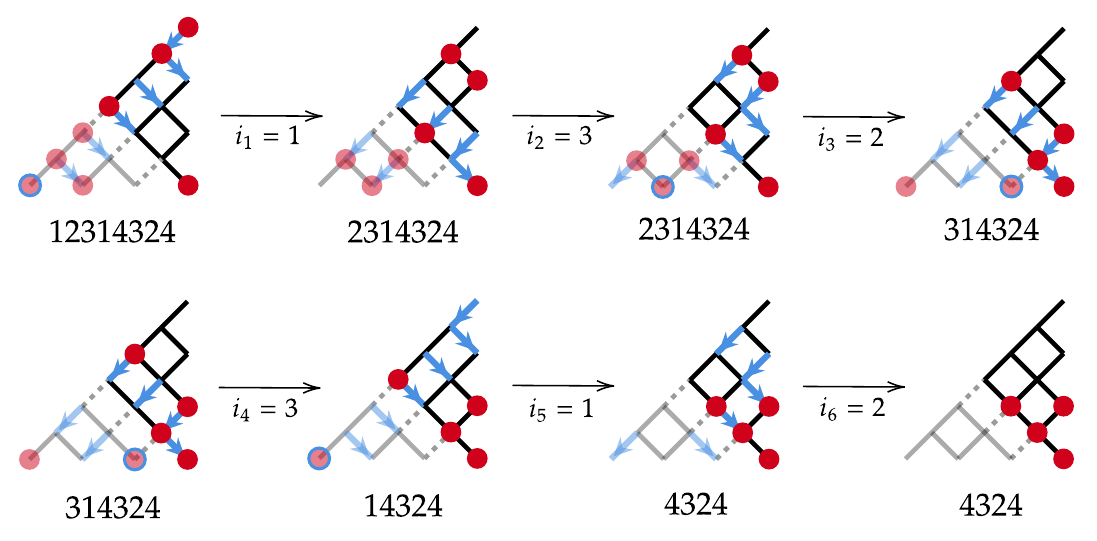}
\caption{The sequence of packings corresponding to $132312$ (a reduced word for the longest element of $W_J$), going from $w=12314324$ to $^Jw=4324$, with $W=C_4$ and $J=\{s_1,s_2,s_3\}$.}
\label{fig:packingsc4}
\end{figure}

\subsection{Properties of minuscule middle orders}

Since minuscule posets are self-dual, minuscule middle orders are self-dual.
\begin{prop}
   All minuscule middle orders on a Weyl group $W$ are self-dual, with an antiautomorphism given by $w\mapsto ww_0$, where $w_0$ is the maximal element of $W$.
\end{prop}
\begin{proof}
    When $s$ is minuscule and $J=\{s\}$, $^JW$ has an antiautomorphism given by $w\mapsto w_0 (^Jw_0)^{-1} ww_0$, where $^Jw_0$ is the longest element of $^JW$. Assuming that $W_J$ has an antiautomorphism given by $w\mapsto w \ \! ^Jw_0 w_0$ ($^Jw_0 w_0$ being the longest element of $W_J$), we obtain that $W$ has an antiautomorphism given by $w=w_J \cdot \! ^Jw\mapsto w_J \ \! ^Jw_0 w_0 \cdot w_0 (^Jw_0)^{-1} \ ^J \! w w_0=ww_0$.
\end{proof}

Nonnesting partitions associated to a Weyl group $W$ with root poset $\Phi^+$ are defined as the antichains of $\Phi_+$, or equivalently as its lower sets. Nonnesting partitions are counted by a generalization of Catalan numbers~\cite[Section 2.7]{Arms09}:
% TO DO ref nombre de Catalan
$$\Cat(W)=\prod_{k=1}^{n-1}\frac{h+d_k}{d_k}$$
where the $(d_k)_{k=1}^n$ are the degrees of $W$ and $h=d_n$ is its Coxeter number.

\begin{prop}
\label{pr:NNsublattice}
    Let $W$ be a Weyl group with root poset $\Phi^+$. Any minuscule middle order on $W$ has a sublattice isomorphic to the lattice of nonnesting partitions $\IdL(\Phi^+)$.
\end{prop}
\begin{proof}
   A minuscule middle order is isomorphic to the lattice of lower sets of $\RPos_{s_{i_1},\dots, s_{i_r}}$, which is a coarsening of the root poset. Hence, lower sets of $\Phi^+$ are lower sets of $\RPos$, with the same meet and join operations as the corresponding middle order.
\end{proof}

Let $\Pos$ be a finite poset, we write $F_m(\Pos,q)$ the generating function for $m$-multichains of lower sets of $\Pos$, in which the coefficient of $q^k$ is the number of elements of rank $k$ in $\IdL(\Pos\times [m])$ where $[m]$ is a chain with $m$ elements.

\begin{prop}
	Let $W$ be a Weyl group with root poset $\Phi_+$. Let $\L$ be a minuscule middle order on $W$ with poset of irreducibles $\RPos$, we have
	$$ F_m(\RPos,q)=\prod_{\alpha\in \Phi^+}\frac{1-q^{m+r_\alpha}}{1-q^{r_\alpha}},$$
	where $r_\alpha-1$ is the rank of $\alpha$ in the root poset, and the number of self-dual lower sets of $\RPos\times [m]$ is equal to $F_m(\RPos,-1)$.
	\label{pr:Fpoly}
\end{prop}
\begin{proof}
	Let $\Pos_\lambda$ be a minuscule poset, we have~\cite[Theorem 6]{Pro84}
	%	$$F_m(\Pos_\lambda,q)=\prod_{\alpha\in \Phi^+}\frac{1-q^{\langle m\lambda+\rho,\alpha^\vee\rangle}}{1-q^{\langle\rho,\alpha^\vee\rangle}}$$
	%	where $\alpha^\vee=2\alpha/\langle \alpha,\alpha\rangle$ and $2\rho=\sum_{\alpha\in \Phi^+}\alpha$. $\langle\rho,\alpha^\vee\rangle-1$ is the rank of $\alpha$ in the root poset, and $\langle\lambda,\alpha^\vee\rangle$ is equal to $1$ if $\alpha\in \Pos_\lambda$ and $0$ if not. Hence, we have
	$$F_m(\Pos_\lambda,q)=\prod_{\alpha\in P_\lambda}\frac{1-q^{m+r_\alpha}}{1-q^{r_\alpha}}.$$
	Since $\RPos$ is the disjoint union of minuscule posets, and since the number of self-dual lower sets of $\Pos_\lambda\times[m]$ is $F_m(\Pos_\lambda,-1)$ (\emph{cf.} \cite[Theorem 4.1]{Stem94}), we obtain the result. 
\end{proof}
This result implies that minuscule middle orders on a given Weyl group share the same number of chains and self-dual chains of a given length (for type $A$, we recover the formulas obtained in Section \ref{sec:chains}).

\section{Middle orders and $\omega$-sorting orders}
\label{sec:sorting}
In \cite{Arm09}, Armstrong defined the $\omega$-sorting order, which we denote by $\Sort_\omega$ for $\omega$ a word in the generators of a Coxeter group $W$. It is a join-distributive lattice on $W_\omega$, \emph{i.e.} the elements of $W$ occurring as subwords of $\omega$. For $w\in W_\omega$, $\sort_\omega(w)$ is the index set of the lexicographically minimal subword of $\omega$ which evaluates to $w$, and $w_1\le w_2$ in $\Sort_\omega$ if $\sort_\omega(w_1)\subset \sort_\omega(w_2)$.
Hence, if $\omega$ contains $w_0$ as a subword, the $\omega$-sorting order is a lattice on the whole group. We will only consider the case where $\omega$ is a reduced word for the maximal element of the group (or of a quotient, \emph{cf.} Definition \ref{def:qsort}). These lattices are refinements of the weak order and coarsenings of the Bruhat order. Two sorting orders induced respectively by reduced words $\omega$ and $\omega'$ are equal if and only if $\omega$ and $\omega'$ are equal up to commutation. 

For any minuscule root $s$ of a Weyl group $(W,S)$, there exists a natural labeling of the minuscule poset $\Pos_s$ by generators in $S$, which is called a \defn{minuscule heap} (see~\cite{RuSh12, Stem96}). For any linear extension $x_1,...,x_t$ of $\Pos_s$ with corresponding labels $s_{k_1},...,s_{k_t}$, the map $\phi:\IdL(\Pos_s)\to \,^JW$ (where $J=S\setminus \{s\}$) that sends a lower set $I$ to the product $s_{k_1}\cdots s_{k_t}$ restricted to labels of elements of $I$ is an isomorphism. Each element of $^JW$ is fully commutative, \emph{i.e.}, all its reduced words are the same up to commutation.
Some of the $\omega$-sorting orders are distributive, and when $W$ is a Weyl group, we show that a minuscule middle order can be constructed as $\omega$-sorting orders where $\omega$ is a product of fully commutative elements. See Figure \ref{fig:minusculeheap} for an example.

\begin{prop}
\label{pr:minussorting}
   Let $(W,S)$ be a Weyl group, and $\L$ be a minuscule middle order on $W$ induced by a permutation $s_{i_1},\dots, s_{i_r}$ of $S$. Let $\omega= \omega_r\omega_{r-1}\dots \omega_2\omega_1$ where for all $1\le k\le r$, $\omega_i$ is a reduced word for the minimal element of the topmost coset in $W_{\{s_{i_{k+1}},...,s_{i_r}\}}\backslash W_{\{s_{i_k},...,s_{i_r}\}}$. Then $\L$ is equal to $\Sort_\omega$.
\end{prop}
\begin{proof}
	It follows immediately from the definition of minuscule middle orders that $\omega$ is a reduced word for $w_0$, and that each $w\in W$ has a unique factorization $w_r\cdots w_1$ with $w_r\cdots w_k\in W_J$ and $w_{k-1}\cdots w_1\in \, ^JW$ where $J=\{s_{i_k},...,s_{i_r}\}$ for all $1\le k \le r$.
	 The $\omega$ we defined is equal to the product of minuscule heaps $s_{k_1}\cdots s_{k_t}$ corresponding to each minuscule poset in $\RPos$, hence we have a bijection between lower sets of $\RPos$ and elements of $W$ seen as subwords of $\omega$. For each $w\in W$, the corresponding subword of $\omega$ is lexicographically minimal among reduced words of $w$ that are subwords of $\omega$. This comes from the fact that subwords of $s_{k_1}\cdots s_{k_t}$ corresponding to elements of $^JW$ are minimal because in a minuscule heap, labels of non-comparable elements commute, and their concatenation is also lexicographically minimal because each element of $W$ has a unique factorization $w_J\cdot \!^Jw$ with $w_J\in W_J$ and $^Jw\in \!^JW$.
\end{proof}

\begin{figure}
\centering
\includegraphics[width=0.7\linewidth]{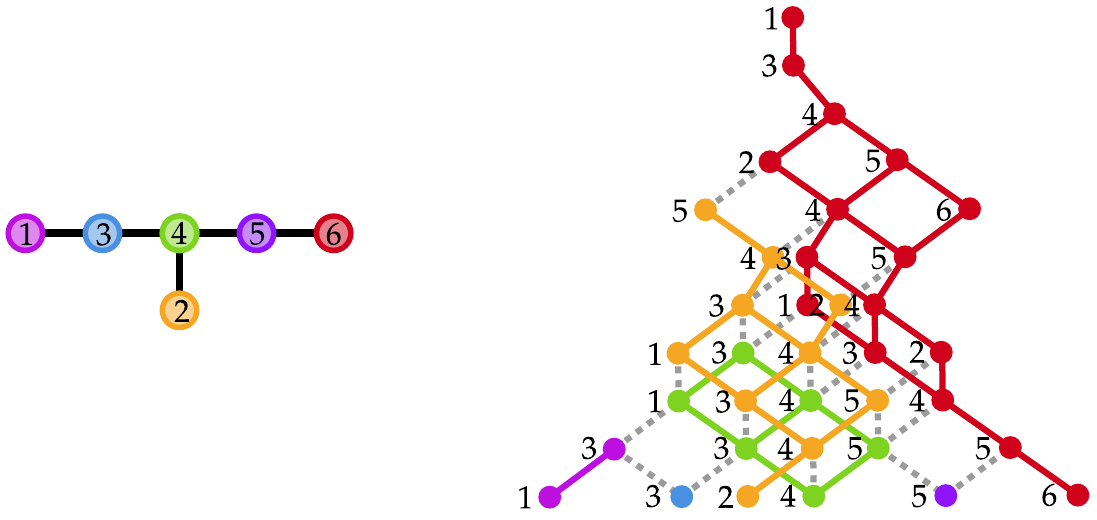}
\caption{\centering The Coxeter diagram of $E_6$, and a partition of its root poset into minuscule posets $\RPos$ induced by the permutation $s_6,s_2,s_4,s_1,s_3,s_5$, with the corresponding minuscule heaps. An example of a sorting word giving a sorting order isomorphic to $\IdL(\RPos)$ is $5|3|13|435143|2435143245|6543214354625431$.}
\label{fig:minusculeheap}
\end{figure}

\begin{rmk}
Remark that for all $w\in W$, the lower set of $\RPos$ corresponding to $w$ is the union of the restrictions of $\pack_{\omega_r\dots \omega_{k+1}}(N(w))$ to $\Pos_{s_{i_k}}$ for $1\le k\le r$ (see Section \ref{subsec:minuspacking} for the definition of $\pack$).
\end{rmk}

We now define sorting orders on parabolic quotients, in an attempt to describe all distributive sorting orders:
\begin{defi}
\label{def:qsort}
   Let $J\subset S$ and $^J\omega$ be a reduced word for $^Jw_0$. Let $\omega$ be any reduced word for $w_0$ containing $^J\omega$ as a suffix. The $^J\omega$-sorting order on $^JW$ is the restriction of the $\omega$-sorting order on $W$ to $^JW$.
\end{defi}
% TODO donner une explication
It is straightforward to show that the definition of the $^J\omega$-sorting order does not depend on the choice of $\omega$, that it is between the weak and Bruhat orders on $^JW$, and that reduced words for $^Jw_0$ that are equal up to commutation gives the same sorting orders.

When proved, the following conjecture could be useful for classifying distributive sorting orders.
%, as even when quotients of the form $^JW$ with $J=S\setminus \{s\}$ are not distributive lattices, their middle orders can be classified easily.
\begin{cnj}
   \label{cnj:distrisort}
   Let $\Sort_\omega(W)$ be a distributive sorting order on a finite Coxeter group $(W,S)$. Then there exists $s\in S$ with $J=S\setminus \{s\}$ such that $\omega$ factorizes as $\omega_J\cdot\!^J\omega$ in $W_J\times \! ^J W$ and $\Sort_\omega(W)$ is the product of distributive sorting orders $\Sort_{\omega_J}(W_J)$ and $\Sort_{^J\omega}(^JW)$.
\end{cnj}

%\begin{exmp}
%    Let $W = B_4$ and $J=\{s_1,s_3,s_4\}$ ($s_1,...,s_4$ is the usual labelling of the generators of $B_4$, where $s_3s_4$ has order $4$). $^Jw_0$ has $3$ reduced words up to commutation (generators are written as their indexes for convenience): $21343213432$, $21324321432$ and $23432123432$. The last two yields distributive sorting orders.
%\end{exmp}

If Conjecture~\ref{cnj:distrisort} were true, a classification of distributive sorting orders on maximal parabolic quotients would give a description of all distributive sorting orders. We give a conjectural characterization of those in type $B$ and $D$, as well as a list of sorting orders in $F_4$ and $H_3$ found by computer experiments.

\begin{cnj}
   \label{cnj:Btype}
Let $W=B_n$ generated by $S=\{s_1,...,s_n\}$ (with the usual labelling of the generators, where $s_{n-1}s_n$ has order $4$), and let $J=S\setminus \{s_k\}$. There are at most two distributive sorting orders on $^JW$ that are given by these reduced words (generators are written as their indexes for convenience):
\begin{itemize}
\item $\omega_{k,1}\omega_{k,2}\dots \omega_{k,k}$ where $\omega_{k,i} = k \ k \!+\!1\dots n \!-\!1\ n\ n \! -\!1 \dots i \!+\!1 \ i $;
\item $\omega_{k,1}\omega_{k+1,2}\dots \omega_{n-1,k-1} \ \omega_{n,1} \omega_{n,2} \dots \omega_{n,k}$ where $\omega_{k,i}=k \ k \!-\!1 \dots i \!+\!1 \ i$.
\end{itemize}
\end{cnj}
Note that if  $J=S\setminus\{s_1\}$ (resp. $J=S\setminus\{s_n\}$), these two distributive sorting orders on $^JW$ coincide, since in this case $^JW$ is a distributive lattice, corresponding to the only minuscule quotient on $B_n$ (resp. $C_n$).

\begin{cnj}
\label{cnj:Bprod}
   For $n\ge 3$, all middle orders on $B_n$ can be obtained as the product of a middle order on $W_J$ with a distributive sorting order on $^JW$, where $J$ is a maximal parabolic subgroup.
\end{cnj}

\begin{table}[H]
	\centering
	\begin{tabular}{cccccccccccc}
	\toprule
	   $n$ & $2$ &$3$ & $4$ & $5$ & $6$ & $7$ & $8$ & $9$ & $10$ & $11$ & $12$  \\ \midrule 
& $2$ & $6$ & $19$ & $63$ & $215$ & $749$ & $2650$ & $9490$ & $34318$ & $125104$ & $459152$  \\ \midrule 
 & $4$ & $8$ & $25$ & $81$ & $273$ & $943$ & $3316$ & $11820$ & $42588$ & $154794$ & $566734$  \\ \bottomrule
	\end{tabular}
	\caption{\centering The conjectured number of distributive sorting orders on $B_n$, and the conjectured total number of middle orders on $B_n$.}
\end{table}

\begin{cnj}
   \label{cnj:Dtype}
   All middle orders (and in particular distributive sorting orders) on $D_n$ are minuscule middle orders.
\end{cnj}
Conjectures \ref{cnj:Btype}, \ref{cnj:Bprod} and \ref{cnj:Dtype} where verified by computer up to $B_5$ and $D_5$. A description of the sorting words corresponding to minuscule middle orders can be found in the proof of Proposition \ref{pr:minussorting}. It would be tempting to conjecture that all middle orders in type $A$, $D$ and $E$ (the simply-laced types) are minuscule, although we were not able to test it in type $E$ due to the size of these posets. This would also mean that there is no middle order on $E_8$, since it has no minuscule middle order.

There are $32$ middle orders on $F_4$, out of which $24$ are distributive sorting orders (see Figure \ref{fig:middleordersf4} in Appendix~\ref{app:2}). These are obtained as products of middle orders on $W_J\cong B_3$ with distributive sorting orders on $^JW$ where $J=\{s_1,s_2,s_3\}$ (resp. $J=\{s_2,s_3,s_4\}$), the last ones being obtained from reduced words for $^Jw_0$ $432132343213234$ and $432132432143234$ (resp. $123423212342321$ and $123423123412321$).

The sorting words giving distributive sorting orders on maximal parabolic quotients on $H_3$ are the following:

\begin{itemize}
\item $1232132321$ for $J=\{s_2,s_3\}$;
\item $2132321232132$ and $2321323213232$ for $J=\{s_1,s_3\}$;
\item $323212321323$, $321323212323$ and $321321321323$ for $J=\{s_1,s_2\}$.
\end{itemize}

Since there are respectively $8$, $1$, and $2$ middle orders on $W_J$ for $J=\{s_2,s_3\}$, $\{s_1,s_3\}$ and $\{s_1,s_2\}$, we obtain $16$ middle orders on $H_3$, out of which $10$ are distributive sorting orders. There are 8 middle orders on $H_3$ that cannot be obtained this way, giving $24$ middle orders in total (see Figure \ref{fig:middleordersh3} in Appendix~\ref{app:2}).

\section{Perspectives}

    In an ongoing work, we show that minuscule middle orders contain subposets of particular interest which we call Jurassian lattices. These lattices are congruence-uniform and extremal, and are conjecturally of size $\Cat(W)$. Their Hasse diagrams seem to be regular of degree $r$ where $r$ is the rank of $W$, and to be the $1$-skeletons of polytopes which are in some cases associahedra. These properties are reminiscent of the well-known Cambrian lattices. Jurassian lattices also seem to contain intervals similar to tilting posets, in the sense that they are obtained by removing facets containing the minimal element and have the same size.
    We also began to study middle orders on affine Coxeter groups, in an attempt to generalize the Young lattice described by Björner and Brenti~\cite{BjBr95}.
    
\section*{Acknowledgements}

This research was driven by computer exploration and computations using the open-source
mathematical software \texttt{Sage}~\cite{sage} and its algebraic
combinatorics features developed by the \texttt{Sage-Combinat}
community~\cite{Sage-Combinat}. I would like to thank Jean-Christophe Novelli and Wenjie Fang for their precious help and for proofreading this work.

%
%\begin{figure}[H]
%	\centering
%	\includegraphics[width=0.7\linewidth]{affine_quotient}
%	\caption{\centering The weak order, the middle order and the Bruhat order on $^JW$ for $W=\tilde A_3$ as defined in Proposition \ref{pr:7}, with window notations and affine codes of inverse permutations.}
%	\label{fig:affinequotient}
%\end{figure}

\bibliography{Middle_order}{}

\begin{thebibliography}{10}

\bibitem{Arms09}
Drew Armstrong.
\newblock Generalized noncrossing partitions and combinatorics of coxeter
  groups.
\newblock {\em Mem. Amer. Math. Soc.}, 202(949), 2009.

\bibitem{Arm09}
Drew Armstrong.
\newblock The sorting order on a {C}oxeter group.
\newblock {\em Journal of Combinatorial Theory, Series A}, 116(8):1285--1305,
  2009.

\bibitem{AuSi18}
Gilles Audemard and Laurent Simon.
\newblock On the {G}lucose {SAT} solver.
\newblock {\em International Journal on Artificial Intelligence Tools},
  27(01):1840001, February 2018.

\bibitem{BaRe17}
Emily Barnard and Nathan Reading.
\newblock Coxeter-bi{C}atalan combinatorics.
\newblock {\em Journal of Algebraic Combinatorics}, 47(2):241–300, August
  2017.

\bibitem{Birk37}
Garrett Birkhoff.
\newblock Rings of sets.
\newblock {\em Duke Mathematical Journal}, 3(3), 1937.

\bibitem{Birk48}
Garrett Birkhoff.
\newblock Lattice {T}heory. revised edition.
\newblock {\em Amer. Math. Soc. Colloq. Publ.}, 25(2), 1948.

\bibitem{BjBr95}
Anders Bj\"{o}rner and Francesco Brenti.
\newblock Affine permutations of type {$A$}.
\newblock {\em The Electronic Journal of Combinatorics}, 3(2), September 1995.

\bibitem{BjBr05}
Anders Björner and Francesco Brenti.
\newblock {\em Combinatorics of Coxeter Groups}.
\newblock Graduate Texts in Mathematics. Springer, April 2005.

\bibitem{BFT24}
Mathilde Bouvel, Luca Ferrari, and Bridget~E. Tenner.
\newblock {Between weak and Bruhat: The middle order on permutations}.
\newblock {\em Graphs Comb.}, 41(2), April 2025.

\bibitem{CFY25}
Federico Campanini, Francesca Fedele, and Emine Yıldırım.
\newblock Lattices of pretorsion classes.
\newblock Preprint,
  \href{http://arxiv.org/abs/2511.19223}{\texttt{arXiv:2511.19223}}, 2025.

\bibitem{ClOv25}
Andrew Claussen and Nicholas Ovenhouse.
\newblock {Mixed dimer models for Euler and Catalan numbers}.
\newblock Preprint,
  \href{http://arxiv.org/abs/2503.11936}{\texttt{arXiv:2503.11936}}, 2025.

\bibitem{CLP02}
Henry Cohn, Michael Larsen, and James Propp.
\newblock The shape of a typical boxed plane partition.
\newblock 2002.

\bibitem{Ehr34}
Charles Ehresmann.
\newblock Sur la topologie de certains espaces homogènes.
\newblock {\em The Annals of Mathematics}, 35(2):396, April 1934.

\bibitem{Grae11}
George Gr\"{a}tzer.
\newblock {\em Lattice Theory: Foundation}.
\newblock Springer Basel, 2011.

\bibitem{HNT05}
Florent Hivert, Jean-Christophe. Novelli, and Jean-Yves Thibon.
\newblock The algebra of binary search trees.
\newblock {\em Theoretical Computer Science}, 339(1):129--165, 2005.
\newblock Combinatorics on Words.

\bibitem{LaSc96}
Alain Lascoux and Marcel-Paul Sch\"{u}tzenberger.
\newblock Treillis et bases des groupes de {C}oxeter.
\newblock {\em Electron. J. Combin.}, 3(2):Research paper 27, approx. 35, 1996.

\bibitem{Macd95}
Ian~G. Macdonald.
\newblock {\em Symmetric Functions and Hall Polynomials}.
\newblock Oxford University Press, 1979.

\bibitem{oeis}
{OEIS Foundation Inc.}
\newblock The {O}n-{L}ine {E}ncyclopedia of {I}nteger {S}equences.
\newblock Published electronically at \url{http://oeis.org}.

\bibitem{PiPo17}
Vincent Pilaud and Viviane Pons.
\newblock Permutrees.
\newblock {\em Electronic Notes in Discrete Mathematics}, 61:987--993, 2017.
\newblock The European Conference on Combinatorics, Graph Theory and
  Applications (EUROCOMB'17).

\bibitem{PiSa19}
Vincent Pilaud and Francisco Santos.
\newblock Quotientopes.
\newblock {\em Bulletin of the London Mathematical Society}, 51(3):406--420,
  2019.

\bibitem{Pro84}
Robert~A. Proctor.
\newblock Bruhat lattices, plane partition generating functions, and minuscule
  representations.
\newblock {\em European Journal of Combinatorics}, 5(4):331--350, 1984.

\bibitem{Prop25}
James Propp.
\newblock Lattice structure for orientations of graphs.
\newblock {\em The Electronic Journal of Combinatorics}, 32(4):P4.26, Nov.
  2025.

\bibitem{Read05bis}
Nathan Reading.
\newblock Lattice congruences, fans and {H}opf algebras.
\newblock {\em Journal of Combinatorial Theory, Series A}, 110(2):237--273,
  2005.

\bibitem{Read06}
Nathan Reading.
\newblock Cambrian lattices.
\newblock {\em Advances in Mathematics}, 205(2):313--353, 2006.

\bibitem{RuSh12}
David~B. Rush and XiaoLin Shi.
\newblock On orbits of order ideals of minuscule posets.
\newblock {\em Journal of Algebraic Combinatorics}, 37(3):545–569, June 2012.

\bibitem{Sage-Combinat}
The {S}age-{C}ombinat community.
\newblock {S}age-{C}ombinat: enhancing {S}age as a toolbox for computer
  exploration in algebraic combinatorics, 2024.
\newblock {\tt https://wiki.sagemath.org/combinat}.

\bibitem{Stem94}
John~R. Stembridge.
\newblock {On minuscule representations, plane partitions and involutions in
  complex Lie groups}.
\newblock {\em Duke Mathematical Journal}, 73(2):469 -- 490, 1994.

\bibitem{Stem96}
John~R. Stembridge.
\newblock {On the fully commutative elements of Coxeter groups}.
\newblock {\em Journal of Algebraic Combinatorics}, 5(4):353–385, October
  1996.

\bibitem{sage}
{The Sage Developers}.
\newblock {S}agemath, the {S}age {M}athematics {S}oftware {S}ystem ({V}ersion
  10.2), 2023.
\newblock {\tt https://www.sagemath.org}.

\bibitem{Tsar90}
Sergei~V. Tsaranov.
\newblock Representation and classification of {C}oxeter monoids.
\newblock {\em European Journal of Combinatorics}, 11(2):189--204, 1990.

\end{thebibliography}
\bibliographystyle{plain}

\appendix

\section{}
\label{app:1}
%TODO expliquer brièvement ce que fait le code
The first function creates clauses in conjunctive normal form (CNF) whose variables are strict Bruhat edges that a set of edges must satisfy in order to form a distributive lattice. For example, let us consider $p$ covered by $a$ and $b$ in the Bruhat order, and $q_1, q_2$ common upper covers to $a$ and $b$. If $e_1=(p,a)$, $e_2=(a,q_1)$, $e_3=(b,q_1)$ and $e_4=(b,q_2)$ are strict Bruhat edges, by Proposition \ref{pr:rhombus} we have $e_1\implies (e_2\wedge e_3)\oplus e_4$. In CNF this is written as $(\neg e_1\vee e_2 \vee e_4)\wedge (\neg e_1\vee e_3 \vee e_4)\wedge (\neg e_1\vee \neg e_2 \vee \neg e_3\vee \neg e_4)$.
\begin{python}
def MiddleOrderClauses(n, side='right'):
    """
    Returns a dictionary D with strict Bruhat edges as keys, giving
    the corresponding variables, and a set E containing CNF clauses
    that a set of edges must satisfy to form a distributive lattice.
    """
    P = Poset((Permutations(n), lambda p, q: p.bruhat_lequal(q)))
    D = {e: k + 1 for k, e in enumerate([(p, q) for p, q in P.cover_relations()
                                         if not p.weak_le(q, side)])}
    E = set()
    for p in P:
        ls = P.upper_covers(p)
        for i in range(len(ls)):
            for j in range(i + 1, len(ls)):
                a, b = ls[i], ls[j]
                lx0 = [-D[(p, k)] for k in [a, b] if not p.weak_le(k, side)]
                La = [[(k, q) for k in [a, b] if not k.weak_le(q, side)]
                      for q in P.common_upper_covers([a, b])]
                for t in cartesian_product(La):
                    lx = lx0.copy()
                    lx.extend(D[e] for e in t)
                    E.add(tuple(lx))
                for c in Combinations(La, 2):
                    lx = lx0.copy()
                    lx.extend(-D[k] for i in range(2) for k in c[i])
                    E.add(tuple(lx))
        ls = P.lower_covers(p)
        for i in range(len(ls)):
            for j in range(i + 1, len(ls)):
                a, b = ls[i], ls[j]
                lx0 = [-D[(k, p)] for k in [a, b] if not k.weak_le(p, side)]
                La = [[(q, k) for k in [a, b] if not q.weak_le(k, side)]
                      for q in P.common_lower_covers([a, b])]
                for t in cartesian_product(La):
                    lx = lx0.copy()
                    lx.extend(D[e] for e in t)
                    E.add(tuple(lx))
                for c in Combinations(La, 2):
                    lx = lx0.copy()
                    lx.extend(-D[k] for i in range(2) for k in c[i])
                    E.add(tuple(lx))
    return D, E

def Ideal(p, w, of=0):
    """
    p = permutation of length n-1
    w = permutation of length n
    Compute the ideal of R_T corresponding to w,
    where T is the binary search tree of p.
    """
    if p == []:
        return []
    w1, w2, s, t, k0 = [], [], 1, 1, p[0]
    Id = []
    for k in w:
        if k > k0:
            w2.append(k)
            s += 1
        else:
            w1.append(k)
            Id.extend((of + t, k0 + j) for j in range(1, s))
            t += 1
    Id.extend(Ideal([k for k in p if k < k0], w1, of=of))
    Id.extend(Ideal([k for k in p if k > k0], w2, of=k0))
    return Id

def main(n):
    """
    Creates a .cnf file containing clauses that must be satisfied if there
    exists a middle order on S_n different than those we constructed.
    """
    D, E = MiddleOrderClauses(n)
    with open("clauses.cnf", "w") as file:
        file.write(f'p cnf {len(set(D.values()))} '
                   f'{len(E) + catalan_number(n - 1)}\n')
        for c in E:
            file.write(' '.join([str(k) for k in c] + ['0']) + '\n')
        # add clauses excluding known middle orders
        for p in Permutations(n - 1, avoiding=[2, 3, 1]):
            c = []
            for e in D:
                if Set(Ideal(p, e[0])).issubset(Set(Ideal(p, e[1]))):
                    c.append(-D[e])
                else:
                    c.append(D[e])
            file.write(' '.join([str(k) for k in c]) + ' 0\n')

\end{python}

\pagebreak
\section{}
\label{app:2}

We give an exhaustive list of posets of irreducible elements of middle orders on $F_4$ and $H_3$, with the number of corresponding middle orders. Posets connected by arrows are dual from each other.

\begin{figure}[H]
\centering
\includegraphics[width=0.7\linewidth]{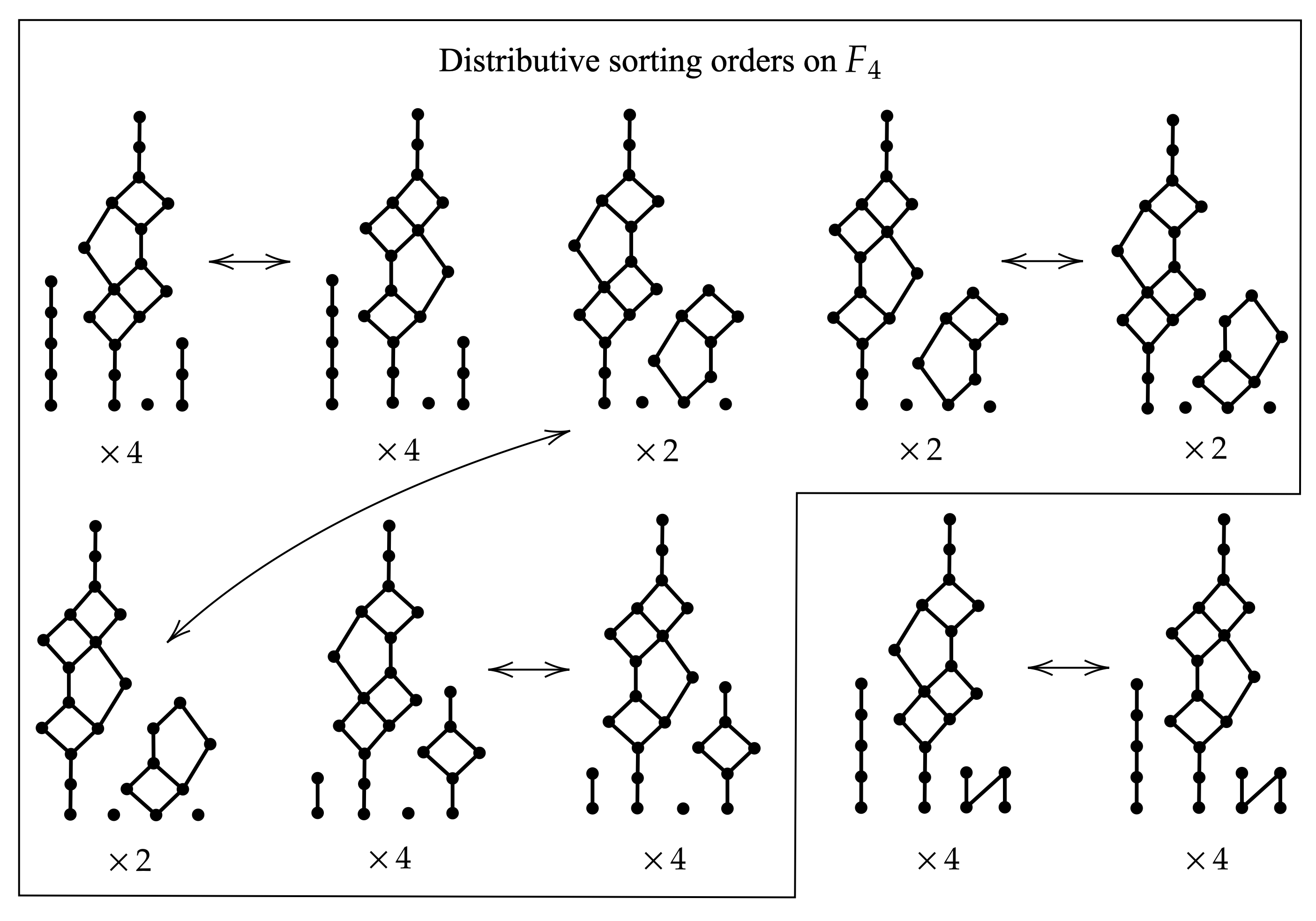}
\caption{Irreducible posets of middle orders on $F_4$.}
\label{fig:middleordersf4}
\end{figure}

\begin{figure}[H]
\centering
\includegraphics[width=0.7\linewidth]{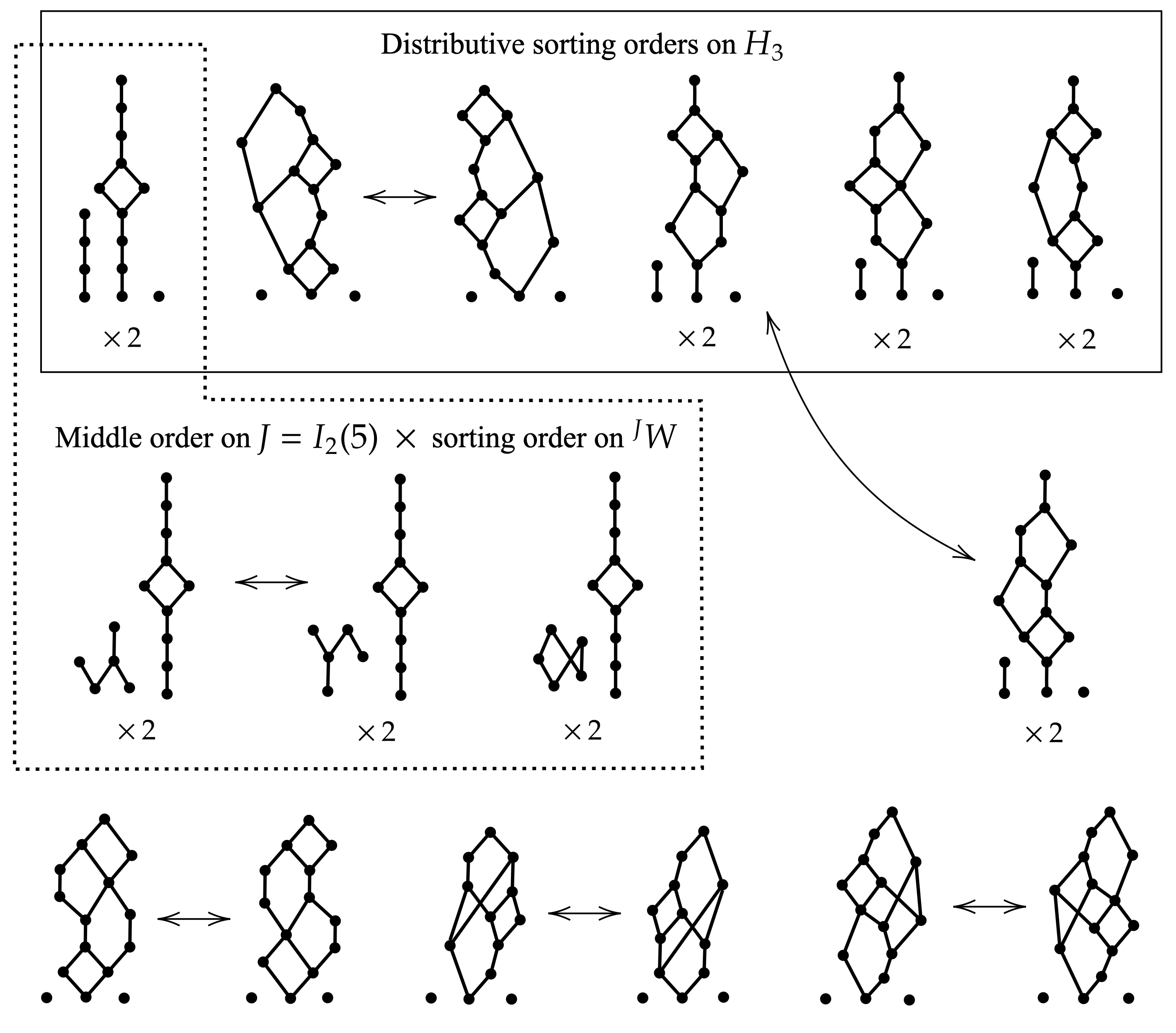}
\caption{Irreducible posets of middle orders on $H_3$.}
\label{fig:middleordersh3}
\end{figure}

\end{document}